\documentclass[11pt]{amsart}
\usepackage[utf8]{inputenc}
\usepackage[T1]{fontenc}
\usepackage{bm,amsmath,amsthm,amssymb}
\usepackage{geometry}
\usepackage{graphicx}
\usepackage{tikz}
\usepackage{amsfonts}
\usepackage{array}
\usepackage{enumerate}
\usepackage{float}
\usepackage{lscape}
\usepackage[english]{babel}
\usepackage[all]{xy}
\usepackage{authblk}
\usepackage{hyperref}
\usepackage{enumitem}

\theoremstyle{plain}
\newtheorem{fact}{Fact}[section]
\newtheorem{theo}[fact]{Theorem}
\newtheorem{lem}[fact]{Lemma}
\newtheorem{defi}[fact]{Definition}

\newtheorem{rmk}[fact]{Remark}
\newtheorem{coro}[fact]{Corollary}

\newcommand{\M}{\mathcal{M}}
\newcommand{\C}{\mathcal{C}}
\newcommand{\E}{\mathcal{E}}
\newcommand{\W}{\mathcal{W}}
\newcommand{\Fib}{\mathcal{F}\textit{ib}}
\newcommand{\Cof}{\mathcal{C}\textit{of}}

\title{Model category structures from rigid objects in exact categories}
\author{Lucie JACQUET-MALO\\
lucie.jacquet.malo@u-picardie.fr\\
LAMFA\\
33 rue Saint Leu\\
80 000 AMIENS\\
FRANCE}

\begin{document}

\maketitle

\begin{center}
\large{\sc{Lucie Jacquet-Malo \\ lucie.jacquet.malo@u-picardie.fr}}
\end{center}

\begin{abstract}
Let $\E$ be a weakly idempotent complete exact category with enough injective and projective objects. Assume that $\M \subseteq \E$ is a rigid, contravariantly finite subcategory of $\E$ containing all the injective and projective objects, and stable under taking direct sums and summands. In this paper, $\E$ is equipped with the structure of a prefibration category with cofibrant replacements. As a corollary, we show, using the results of Demonet and Liu in \cite{DL}, that the category of finite presentation modules on the costable category $\overline{\M}$ is a localization of $\E$. We also deduce that $\E \to \mathrm{mod}~\overline{\M}$ admits a calculus of fractions up to homotopy. These two corollaries are analogues for exact categories of results of Buan and Marsh in \cite{BM2}, \cite{BM1} (see also \cite{Be}) that hold for triangulated categories.

If $\E$ is a Frobenius exact category, we enhance its structure of prefibration category to the structure of a model category (see the article of Palu in \cite{Palu} for the case of triangulated categories). This last result applies in particular when $\E$ is any of the Hom-finite Frobenius categories appearing in relation to cluster algebras.
\end{abstract}




\tableofcontents

\section*{Introduction}

Homotopical algebra was first introduced by Quillen in \cite{Q}. It appears in many subjects, such as geometry (for example the study of $G$-spaces in \cite{DK}, spectra in \cite{BF}) or algebra (the study of commutative rings in \cite{Q2} or simlplicial groups in \cite{Q3}). Its aim was to axiomatize the structure of the category of topological spaces that is relevant for doing homotopy, thus allowing to make use of different categories (such as sSet) and to compare different homotopy theories. This axiomatization not only covers homotopy, but also homological algebra, whence the name of homotopical algebra.

In a seemingly unrelated direction, some subquotients of triangulated categories play a specific role in the theory of additive categorification of cluster algebras.

Let $T$ be a cluster-tilting object in a cluster category $\C$ with shift functor $\Sigma$. Buan, Marsh and Reiten proved in \cite{BMR} that the ideal quotient $\C/(\Sigma T)$ is equivalent to the category $\mathrm{mod} \mathrm{End}_{\C}(T)^{\mathrm{op}}$ of finitely presented modules over the endomorphism algebra of $T$. This result was then generalized by Keller and Reiten in \cite{KR} (to the case of $m$-cluster-tilting objects), Koenig and Zhu in \cite{KZ} (to the case of triangulated categories without any Calabi-Yau assumption) and Iyama and Yoshino in \cite{IY} : If $R$ is a rigid object in a Hom-finite triangulated category with a Serre functor, let $R*\Sigma R$ be the full subcategory of $\C$ whose objects $X$ appear in some triangle $R_1 \to R_0 \to X \to \Sigma R_1$, with $R_1, R_0 \in \mathrm{add} R$. Then the functor $\C(R,-)$ induces an equivalence of categories \[R*\Sigma R/(\Sigma R) \to \mathrm{mod} \mathrm{End}_{\C}(R)^{\mathrm{op}}.\]

In \cite{BM2}, Buan and Marsh proved that $R*\Sigma R/(\Sigma R)$ can be described as a localization of $\C$ (at the class of all morphisms $f$ such that $\C(R,f)$ is an isomorphism in $\mathrm{mod} \mathrm{End}_{\C}(R)^{\mathrm{op}}$). They further proved in \cite{BM1}, that this localization admits a calculus of fractions at the level of the category $\C/(\Sigma R)$.

Finally, Palu in \cite{Palu} showed that there was almost a model structure on $\mathcal{C}$ such that the homotopy category of $\mathcal{C}$ was the localization. To be accurate, he showed that $R*\Sigma R/(\Sigma R)$ was equivalent to the category of cofibrant and fibrant objects up to homotopy.

Our aim in this paper is to prove a similar result for exact categories. In doing so, we provide a conceptual explanation for the existence of the calculus of fractions. Demonet and Liu in \cite{DL} have shown a theorem similar to that of the first result for triangulated categories: Let $\E$ be an exact category with enough projective and injective objects. Let $\M$ be a full subcategory of $\E$ containing the injectives, and contravariantly finite. Let \[\mathrm{pr} \M = \{ X \in \E, \exists M_1,M_0 \in \M, 0 \to M_1 \to M_0 \to X \to 0 \}\] and \[\mho\M = \{ X \in \E, \exists M \in \M, I \in \mathrm{Inj}, 0 \to M \to I \to X \to 0 \}.\] Demonet and Liu proved that there is an equivalence of categories \[\mathrm{pr}\M/\mho M \to \mathrm{mod}~\overline{\M}\] via the functor \[\begin{array}{ccccc}
G & : & \E & \to & \mathrm{Mod} \overline{\M} \\
 & & X & \mapsto & \overline{\E}(-,X)/\overline{\M} \\
\end{array}.\] However, there is no known description of this subquotient as a localization of $\E$. Our aim in this paper is to give one in the case where $\E$ is an exact, then a Frobenius category. We are going to show that, under some technical assumptions the localization of some class of weak equivalences is equivalent to $\mathrm{mod}~\overline{\M}$, and that there exists a (homotopy) calculus of fractions.

More precisely, we show the following results: we start with the particular case of a Frobenius category

\begin{theo}\label{th:modfrob}
Let $\E$ be a Frobenius category. Let $\M$ be a full subcategory of $\E$ containing the injectives, and contravariantly finite and let $\W$ be the class of those morphisms whose image under the functor $G$ is an isomorphism. Then, there exist two classes of morphisms, $\Fib$ and $\Cof$, forming a model structure on $\E$. More precisely, all the objects are fibrant, and an object is cofibrant if and only if it belongs to $\mathrm{pr}\M$.
\end{theo}

We note however that one of the factorizations only exists for morphisms with cofibrant domain.

These results apply to the following classes of examples:
Geiss, Leclerc and Schröer studied module categories over the preprojective algebras in the Dynkin case in \cite{GLS1}, in the non-Dynkin case in \cite{GLS2} and a nice survey can be found in \cite{GLS3}.

Another class of examples is given by the categories of Cohen-Macaulay modules over simple hypersurface singularities studied by Burban, Iyama, Keller and Reiten in \cite{BIKR} (their stable categories also appear in the article of Buan, Palu and Reiten, \cite{BPR}).

As a consequence, results due to Quillen in \cite{Q} can be applied in order to obtain the following:

\begin{coro}\label{coro:pref}
Let $\E$ be a weakly idempotent complete Frobenius category. Let $\M$ be a subcategory of $\E$ which is rigid and contravariantly finite, containing all the injective objects.

Let $\sim$ be the homotopy relation on $\E$ given by the model structure of Theorem \ref{th:modfrob}. Then two morphisms $f$ and $g$ are homotopic if and only if $f-g$ factorizes through $\mho \M$. Let $\mathrm{Ho}~\E$ be the localization of the quasi-isomorphisms of $\E$ at the class $\W$ of weak equivalences. Then there is an equivalence of categories \[\mathrm{pr}\M/(\mho\M) \simeq \mathrm{Ho}~\E.\]
From the equivalence of categories of Demonet and Liu in \cite{DL}, there is an equivalence of categories $\mathrm{Ho}~\E \to \mathrm{mod}~\overline{\M}$.
\end{coro}

The previous results on Frobenius exact categories can be seen as (non-trivial) adaptations of the results of Palu in \cite{Palu} that hold for triangulated categories. However, we also study the more general of (non-Frobenius) exact categories. This situation turns out to be more complicated to handle: Indeed, it seems difficult to hope for a model structure in that case. However, they carry a prefibration structure with cofibrant replacements. It is sufficient to prove localization results analogous to those in the Frobenius setting.

\begin{theo}\label{th:pref}
Let $\E$ be a weakly idempotent complete exact category with enough injective and projective objects. Assume that $\M \subseteq \E$ is a rigid, contravariantly finite subcategory of $\E$ containing all the injective objects, and stable under taking direct sums and summands and such that $\mho \M$ is contravariatly finite.

Then there exist two classes of morphisms $\Fib$ and $\W$ (the same weak equivalences as in the Frobenius case) such that $(\E,\Fib,\W)$ has a prefibration structure in the sense of Anderson-Brown-Cisinski (see \cite{R} for a deeper study). All the objects are fibrant, and an object is cofibrant if and only if it belongs to $\mathrm{pr}\M$. Moreover, under the assumption that $\M$ contains the projective objects, there exist some cofibrant replacements for any object of $\E$.
\end{theo}

This theorem is sufficient to prove Corollary \ref{coro:pref} in the case where $\E$ is only equipped with an exact structure.

\begin{coro}
We deduce from this that there is a homotopy calculus of fractions in the sense of Radulescu-Banu in \cite{R}.
\end{coro}

The latter result gives a theoretical interpretation to the results of Buan and Marsh in \cite{BM1} concerning the calculus of fractions.

The paper is organised as follows: In the first part, we get interested in the exact case. Section $1$ is devoted to the study of $\mathrm{pr}\M$. Then in section $2$, we focus on weak equivalences and fibrations. In section $3$ we give explicit factorizations. Section $4$ is devoted to the cofibrant replacements and the homotopy relation. The last two sections of this part consist in the proof of the prefibration structure (Theorem \ref{th:pref}) and the theorem of Quillen (\ref{coro:pref}).

In the second part, we consider the particular case of Frobenius categories. In section $1$, we give additionnal properties of $\mathrm{pr}\M$ that hold in this more specific setup. Then, in section $2$, we show a caracterization of weak equivalences. Section $3$ deals with fibrations and cofibrations, and in section $4$, we find the second factorization. Finally, in section $5$, we show the model structure on the Frobenius category.

\section{Preliminaries}

\subsection{Conventions and results on exact categories}

In this section, we choose as a reference the article of Bulher \cite{B}.

\begin{defi}
Let $\E$ be an additive category. A short exact sequence $(i, p)$ in $\E$ is a pair of composable morphisms $\xymatrix{A \ar^i[r] & A' \ar^p[r] & A''}$
such that $i$ is a kernel of $p$ and $p$ is a cokernel of $i$. If a class $C$ of short exact sequences on $\E$ is fixed, an inflation is a morphism $i$ for which there exists a morphism $p$ such that $(i, p) \in C$. Deflations are defined dually.

An exact structure on $\E$ is a class $C$ of short exact sequences, which is closed under isomorphisms and satisfies the following axioms:
\begin{description}
\item[E0] The identity morphism is an inflation.
\item[E0'] The identity morphism is a deflation.
\item[E1] The class of inflations is closed under composition.
\item[E1'] The class of deflation is closed under composition.
\item[E2] The push-out of an inflation along an arbitrary morphism exists and yields
an inflation.
\item[E2'] The pull-back of a deflation along an arbitrary morphism exists and yields
a deflation.
\end{description}
\end{defi}

\begin{defi}
We say that $\E$ has enough injective objects if any object $A$ of $\E$ appears in a short exact sequence $A \to I \to B$, where $I$ is injective. Dually, we define what "having enough projective" means.

It is said that $\E$ is a Frobenius category if there are enough projective and injective objects, and if they coincide.
\end{defi}

\begin{lem}[Buhler, \cite{B}]
Let \[ \xymatrix{A~ \ar@{>->}^f[r] \ar_a[d] & B \ar^b[d] \\
C~ \ar@{>->}_g[r] & D}\] be a commutative square, where the horizontal arrows are inflations. The following assertions are equivalent:
\begin{itemize}
\item The square is a pushout
\item The sequence $A \to B \oplus C \to D$ is short exact
\item The square is bicartesian, that is, a pushout and a pullback
\end{itemize}
\end{lem}

\begin{lem}[Demonet-Liu, \cite{DL}]
If $0 \to X \to Y \to Z \to 0$ and $0 \to Y \to W \to V \to 0$ are two short exact sequences, then there is a commutative diagram of short exact sequences where the upper-right square is both a push-out and a pull-back:
\[ \xymatrix{ & & 0 \ar[d] & 0 \ar[d] & \\
0 \ar[r] & X \ar@{=}[d] \ar[r] & Y \ar[r] \ar[d] & Z \ar[d] \ar[r] & 0 \\
0 \ar[r] & X \ar[r] & W \ar[r] \ar[d] & U \ar[d] \ar[r] & 0 \\
& & V \ar@{=}[r] \ar[d] & V \ar[d] & \\
& & 0 & 0 &} \]
\end{lem}

\subsection{Weakly idempotent complete categories}
In this part, we explain the notion of a weakly idempotent complete category, a concept we need all throughout the paper.

\begin{defi}
Let $\mathcal{C}$ be an additive category. Then $\mathcal{C}$ is said to be weakly idempotent complete when any retraction has a kernel (which is equivalent to the fact that any section has a cokernel).
\end{defi}

We also have the following characterization:

\begin{theo}[\cite{B}]
If $\mathcal{C}$ is exact, then the category $\mathcal{C}$ is weakly idempotent complete if and only if any retraction is a deflation (or equivalently, any section is an inflation).
\end{theo}

\begin{coro}[\cite{B}]
If $\mathcal{C}$ is exact, then the category $\mathcal{C}$ is weakly idempotent complete if and only if:

for two composable morphisms $f$ and $g$, if $g \circ f$ is a deflation, then so is $g$.
\end{coro}

We will freely make use of this latter characterization throughout the paper.

\subsection{An equivalence of categories}
We recall the definition of a rigid subcategory.
\begin{defi}
Let $\E$ be an exact category. Then, a subcategory $\M$ of $\E$ is said to be rigid if $\mathrm{Ext}^1_{\E}(\M,\M)$ is zero.
\end{defi}

\begin{defi}
We call by $\mathrm{Mod} ~\overline{\M}$ the category of modules over $\M$. We also call by $\mathrm{mod} ~\overline{\M}$ the category of modules over $\M$ of finite type.
\end{defi}

Demonet and Liu showed in \cite{DL} that there is an equivalence of category that we will try and describe as a localization.
\begin{theo}[\cite{DL}]
Let $\mathcal{E}$ be an exact category. Let $\M$ be a full rigid subcategory which contains the injective objects. Let \[ \mathrm{pr} \M = \{ X \in \E, \exists M_1,M_0 \in \M, 0 \to M_1 \to M_0 \to X \to 0 \} \] and
\[ \mho\M = \{ X \in \E, \exists M \in \M, I \in \mathrm{Inj}, 0 \to M \to I \to X \to 0 \}. \] Let $G$ be the following functor:
\[\begin{array}{ccccc}
G & : & \E & \to & c \\
 & & X & \mapsto & \overline{\E}(-,X)|\overline{\M} \\
\end{array}\]
Then, the functor $G$ induces the following equivalence:
\[\mathrm{pr} \M/\mho \M \simeq \mathrm{mod} ~\overline{\M}. \]
\end{theo}

\subsection{Model categories}
In this part, we recall the definition of a model category due to Quillen in \cite{Q}. For a gentle introduction to model categories, see Dwyer and Spalinski in \cite{DS}. For a deeper study, see Hovey in \cite{Ho}.
\begin{defi}
Let $\mathcal{C}$ be a category equipped with a class W of morphisms called weak-equivalences. It is said that there is a model structure on $\mathcal{C}$ when there exist three classes of morphisms called $(\W, \Fib, \Cof)$ such that:
\begin{description}
\item[MC1] Finite limits and colimits exist in $\mathcal{C}$.
\item[MC2] $\W$ has the "two-out-of-three" property, meaning that, for two composable maps $g$ and $f$, if two of the three maps $f$, $g$ or $gf$ are in $\W$, so is the third.
\item[MC3] The three classes $\Fib$, $\Cof$ and $\W$ are stable under retracts.
\item[MC4] For a commutative diagram \[ \xymatrix{A \ar^a[r] \ar_h[d] & X \ar^f[d] \\ B \ar_b[r] \ar@{.>}[ur] & Y}\] a lift exists if: either $h$ is a cofibration and $f$ is an acyclic fibration (meaning that $f \in \Fib \cap \W$) or $h$ is an acyclic cofibration and $f$ is a fibration.
\item[MC5] Any morphism can be factored in two ways: as a cofibration followed by an acyclic fibration, and as an acyclic cofibration followed by a fibration.
\end{description}
\end{defi}

\subsection{The homotopy category}
In his book, Quillen proved that the localization $\mathrm{Ho}~\mathcal{C}$ of a model category $\mathcal{C}$ obtained by inverting all the weak equivalences is equivalent to the quotient category $\mathcal{C}_{cf}/\sim$ of the cofibrant and fibrant objects by the homotopy equivalence. This proves that the localization is indeed a category. For a detailed proof of this result, see also Hovey in \cite[1.2.9]{Ho}.

\begin{lem}[Quillen]
Let $\mathcal{C}$ be a model category. If $X$ and $Y$ are fibrant and cofibrant, then the homotopy relation is an equivalence relation on $\mathcal{C}(X,Y)$, compatible with the composition.
\end{lem}

The consequence is that the category of fibrant and cofibrant objects up to homotopy is well-defined.

\begin{theo}[\cite{Q}]\label{th:ho}
Let $\mathcal{C}$ be a model category. Let $\mathcal{C}_{cf}$ be the full subcategory of fibrant cofibrant objects. Let $\sim$ be the homotopy relation. Let $\mathrm{Ho}~\mathcal{C}$ be the localization of $\mathcal{C}$ obtained by inverting all the weak equivalences. Then we have \[\mathrm{Ho}~\mathcal{C} \simeq \mathcal{C}_{cf}/\sim.\]
\end{theo}

\begin{rmk}\label{rmk:ho}
In his proof, Hovey only needs the existence of the second factorization on cofibrant objects. So, as we are only going to show this factorization on cofibrant objects, this is sufficient to use this theorem.
\end{rmk}

\subsection{Constructing model structures}
There is a theorem we can use in order to show that a category is equipped with a model structure. But let first recall a definition.

\begin{defi}
Let $f:A \to B$ and $g:X \to Y$ be two morphisms in a category $\C$. We say that $f \square g$ if, for any commutative square \[\xymatrix{A \ar^a[r] \ar_f[d] & X \ar^g[d] \\
B \ar_b[r] \ar@{.>}[ur] & Y}\]
there exists an arrow from $B$ to $X$ such that both triangles commute.

If $\mathcal{A}$ is a class of morphisms, we write \[^\square\mathcal{A}=\{f \text{ morphism of } \C \text{ such that } f \square g,~\forall g \in \mathcal{A}\}.\]
Dually, we write  \[\mathcal{A}^\square=\{g \text{ morphism of } \C \text{ such that } f \square g,~\forall f \in \mathcal{A}\}.\]
\end{defi}

\begin{theo}[\cite{Ho}]\label{th:hov}
Let $\mathcal{C}$ be a category. Let $\W$, $\Fib$, $\Cof$ be three classes of morphisms, where there exists $I$ (respectively $J$) such that $\Fib=J^\square$ (respectively \\ $\Cof=^\square (I^\square)$). Then there is a model structure on $\mathcal{C}$ if and only if:
\begin{enumerate}[label=(\roman*)]
\item The class of morphisms $\W$ has the "two-out-of-three" property, and is stable under retracts.
\item We have $J^\square \cap \W=I^\square$ (elements in this class are called acyclic fibrations).
\item We have $^\square(J^\square) \subseteq \W \cap~^\square(I^\square)$ (elements in this class are called acyclic cofibrations).
\item Any morphism can be factored through an acyclic cofibration followed by a fibration. Any morphism can be factored through a cofibration followed by an acyclic fibration.
\end{enumerate}
\end{theo}

\begin{rmk}
The previous theorem is usually stated in a slightly different manner. Factorizations are not required to be given a priori ; instead, one requires that the domains of $I$ (respectively $J$) 
be small relative to $I$-cell (respectively $J$-cell), where, for a class of morphisms $A$, $A$-cell is the collection of $A$-cell complexes, which is a transfinite composition of pushouts of elements of $A$. The factorizations are then given by the small object argument.

Here, we cannot apply the small object argument, because, in most of our examples (for example, see \cite{BIKR}), we do not have all the colimits. We therefore explicitly compute the factorizations without using the small object argument.
\end{rmk}

\subsection{About a weak model structure on triangulated categories}
In his paper \cite{Palu}, Palu has shown that any covariantly finite rigid subcategory of a triangulated category gives rise to a weak model structure. This inspires the idea of the proof in this article, but all proofs are independent.

\begin{theo}[Palu, \cite{Palu}]
Let $\mathcal{C}$ be a triangulated category with a covariantly finite, rigid subcategory $\mathcal{T}$.

There is a left-weak model category structure on $\mathcal{C}$, where:
\begin{itemize}
\item All objects are fibrant.
\item An object is cofibrant if and only if it belongs to $\mathcal{T} * \Sigma \mathcal{T}$.
\item Weak equivalences $\W$ compose the class of morphisms $X \to Y$ such that, for any triangle $Z \to X \to Y \to \Sigma Z$, both other morphisms belong to the ideal $(\mathcal{T}^\perp)$.
\item Two morphisms are homotopic if and only if their difference factors through $\mathrm{add} \mathcal{T}$.
\end{itemize}

It means that $\mathcal{C}$ contains three classes of morphisms $\W$, $\Fib$, $\Cof$ such that:
\begin{enumerate}
\item Pullbacks of acyclic fibrations along deflations exist and are acyclic fibrations.
\item The class $\W$ has the "two-out-of-three" property.
\item The three classes $\W$, $\Fib$, $\Cof$ contain all identities, are stable under retracts and composition.
\item We have $\W \cap \Cof \subseteq ^\square \Fib$ and $\Cof \subseteq ^\square(\W \cap \Fib)$
\item Any morphism can be factored through an acyclic cofibration followed by a fibration. Any morphism with cofibrant domain can be factored through a cofibration followed by an acyclic fibration.
\end{enumerate}
\end{theo}

\subsection{Prefibration categories}

In his paper \cite{R}, Radulescu-Banu works on prefibration categories.
\begin{defi}
An ABC prefibration category (ABC stands for Anderson-Brown-Cisinski) consists of a category $\E$, with two classes of maps, the fibrations $\Fib$ and the weak equivalences $\W$ satisfying the following axioms:

\begin{description}
\item[F1] $\E$ has a final object which is fibrant. Fibrations are stable under composition. All isomorphisms are weak equivalences, and all isomorphisms with fibrant codomain are acyclic fibration (it means fibrations which are also weak equivalences).
\item[F2] $\W$ has the "two-out-of-three" property, meaning that, for two composable maps $g$ and $f$, if two of the three maps $f$, $g$ or $gf$ are in $\W$, so is the third.
\item[F3] Both classes $\Fib$ and $\W$ are stable under retracts.
\item[F4] For a diagram \[ \xymatrix{A \ar@{.>}^a[r] \ar@{.>}_h[d] & X \ar^f[d] \\ B \ar_b[r] & Y}\] the pullback exists in $\E$. Moreover, if $f$ is a fibration (respectively an acyclic fibration), then $h$ is a fibration (respectively an acyclic fibration).
\item[F5] Any morphism $f:A \to B$, with $B$ fibrant, can be factored through a weak equivalence followed by a fibration.
\end{description}
\end{defi}

He also shows the useful theorem:

\begin{theo}[Radulescu-Banu,\cite{R}, Theorem 6.4.2]
If $\E,\W$ is a category pair satisfying the following:
\begin{enumerate}
\item The two-out-of-three axiom.
\item For a diagram \[ \xymatrix{A \ar@{.>}^a[r] \ar@{.>}_h[d] & X \ar^f[d] \\ B \ar_b[r] & Y}\] the pullback exists in $\E$. Moreover, if $f$ is a weak equivalence, then $h$ is a weak equivalence.
\item For any maps $\xymatrix{A \ar@<1ex>^f[r] \ar@<-1ex>_g[r] & B \ar^t[r] & B'}$ with $t \in \W$, and $t \circ f=t \circ g$, there exists $t':A' \to A \in \W$ such that $f \circ t'=g \circ t'$.
\end{enumerate}
Then, we have the following results:
\begin{enumerate}
\item Each map $h:A \to B$ in the homotopy category can be written as a right fraction $f \circ s^{-1}$, with $s \in \W$.
\item Two fractions $f \circ s^{-1}$ and $g \circ s^{-1}$ are equal in the homotopy category if and only if there exist weak equivalences $s'$, $t'$, as in the diagram below, and such that $s \circ s'=t \circ t'$ and $f \circ s'=g \circ t'$:

\[ \xymatrix{ & & A''' \ar_{s'}[dl] \ar^{t'}[dr] & & \\
& A' \ar_s[dl] \ar_f[drrr] & & A'' \ar^g[dr] \ar_t[dlll] & \\
A & & & & B}\]

\end{enumerate}
\end{theo}

\subsection{Some examples of such categories}

As it has been said in the introduction, there are several examples of applications of this article.

\begin{defi}
Let $A$ be a finite-dimensional algebra over an algebraically closed field $K$. Let $Q=(Q_0,Q_1,s,t)$ be a Dynkin quiver. Let $\overline{Q}$ be the quiver obtained from $Q$ by adding, for each arrow $\alpha$ of $Q$ from $i$ to $j$, an arrow $\alpha^*$ from $j$ to $i$. Let $K\overline{Q}$ be the path algebra over $\overline{Q}$. Then the preprojective algebra associated with $Q$ is defined as \[ \Lambda_Q = K\overline{Q}/\mathcal{I}\] where $\mathcal{I}$ is the ideal generated by the element\[ \sum_{\alpha \in Q_1} (\alpha^* \circ \alpha-\alpha \circ \alpha^*).\]
\end{defi}

\begin{lem}
Let $Q$ and $Q'$ be two quivers of the same Dynkin type. Then $\Lambda_Q$ is isomorphic to $\Lambda_{Q'}$.
\end{lem}

In the Dynkin case, the category $\mathrm{mod}\Lambda$ is Frobenius, Hom-finite, stably $2$-Calabi-Yau and has cluster-tilting objects.

\begin{defi}
Let $Q$ be a finite connected quiver with $n$ vertices and without oriented cycles, and let $\Lambda$ be its associated preprojective algebra. As $KQ$ is a subalgebra of $\Lambda$, we denote by $\pi_Q:\mathrm{mod}(\Lambda) \to \mathrm{mod}(KQ)$ as the restriction functor.

Let $M$ be a $KQ$-module. Then $M$ is called terminal if the following hold:
\begin{itemize}
\item $M$ is preinjective
\item If $X$ is an indecomposable $KQ$-module, with $\mathrm{Hom}_{KQ}(M,X) \neq 0$, then $X \in \mathrm{add}~M$.
\item Any indecomposable injective $KQ$-module belongs to $\mathrm{add}~M$.
\end{itemize}
\end{defi}

\begin{theo}[Geiss, Leclerc, Shröer, \cite{GLS2}, Theorem 2.1]
Let $\mathcal{C}_M$ be the subcategory of all $\Lambda$-modules whose image under $\pi_Q$ belongs to $\mathrm{add}~M$. Then the following holds:
\begin{itemize}
\item The category $\C_M$ is Hom-finite.
\item The category $\mathcal{C}_M$ is Frobenius with $n=|Q_0|$ indecomposable projective objects.
\item The stable category of $\mathcal{C}_M$ is $2$-Calabi-Yau.
\end{itemize}
\end{theo}

We can consequently apply the results of part II.

\section{The case of exact categories}
\subsection{Study of the properties of $\mathrm{pr}\M$}

We show some preliminary lemmas which will be used in section \ref{sec:pref} in order to associate a prefibration structure with a given rigid subcategory.

We note that Lemma 1.1, which is an analogue for exact categories of Lemma 3.3 shown by Buan and Marsh in \cite{BM2}, might be of independent interest.

Let $\E$ be a weakly idempotent complete exact category with enough injective and projective objects. Assume that $\M \subseteq \E$ is a rigid, contravariantly finite subcategory of $\E$ containing all the injective objects, and stable under taking direct sums and summands. Let
\[ \mathrm{pr} \M = \{ X \in \E, \exists M_1,M_0 \in \M, 0 \to M_1 \to M_0 \to X \to 0 \} \] and
\[ \mho\M = \{ X \in \E, \exists M \in \M, I \in \mathrm{Inj}, 0 \to M \to I \to X \to 0 \}. \]

Let \[\begin{array}{ccccc}
G & : & \E & \to & \mathrm{Mod} \overline{\M} \\
 & & X & \mapsto & \overline{\E}(-,X)/\overline{\M} \\
\end{array}\] which induces the following equivalence of categories \[ \mathrm{pr}\M/\mho \M \simeq \mathrm{mod}~\overline{\M}. \]

For more details, see the article of Demonet and Liu, \cite{DL}.

In the following lemma, we prove that if M is contravariantly finite, then so is $\mathrm{pr}\M$, provided that $\M$ also contains all projective objects.

This lemma tells us that we can replace each object by a cofibrant replacement.

\begin{lem}\label{lem:prmapp}
Assume that the rigid subcategory $\M$, contains all injectives and all projectives. Then, for any X in E, there exist A in $\mathrm{pr}\M$ and a right $\mathrm{pr}\M$ approximation $A \to X$.
\end{lem}

\begin{proof}
Let $X \in \E$. Let $M_0 \to X$ be an $\M$-approximation of $X$. Since $\E$ is weakly idempotent complete, with enough projectives and $\M$ contains the projective objects, the morphism $M_0 \to X$ is a deflation. Let \[0 \to K_0 \to M_0 \to X \to 0\] be the associated short exact sequence. Similarly, let \[0 \to K_1 \to M_1 \to K_0 \to 0\] be a short exact sequence coming from an $\M$-approximation of $K_0$. Then we have the following diagram:
	
\[ \xymatrix{& 0 \ar[d] & & & \\
	& K_1 \ar^\beta[d] & & & \\
	& M_1 \ar^\alpha[d] \ar@{=}[r] & M_1 \ar[d] & & \\
	0 \ar[r] & K_0 \ar^b[r] \ar[d] & M_0\ar^a[r] & X \ar[r] & 0 \\
	& 0 & & &} \]
	
We have $a \circ b \circ \alpha=0$. Let $A$ be the push-out of the square:
	
\[ \xymatrix{
	M_1 \ar^{b\alpha}[r] \ar[d] & M_0 \ar^r[d] \ar@/^1pc/^a[ddr] & \\
	I_{M_1} \ar[r] \ar@/_1pc/^0[drr] & A \ar^{\exists \varphi}[dr] & \\
	& & X} \]
Then there exists a morphism $\varphi:A \to X$ such that $\varphi \circ r=a$ and the other triangle commutes.
	
We have $A \in \mathrm{pr} \M$. Indeed, we have the following short exact sequence: \[0 \to M_1 \to I_{M_1} \oplus M_0 \to A \to 0.\]
	
Moreover, $A \to X$ is an approximation. Indeed, let $i:B \to X$ be a morphism, with $B \in \mathrm{pr} \M$. Let us show that there exists $B \to A$ which makes the triangle commute.
	
Let \[0 \to M_1' \to M_0' \to B \to 0\] be a short exact sequence with $M_0,M_1 \in \M$. We have the following diagram:
\[ \xymatrix{
	& & & M'_1 \ar[d] \ar[dl] \ar@{.>}_{\exists \delta}[dlll]\\
	M_1 \ar_{\iota_{M_1}}[d] \ar^{ab\alpha}[drr] \ar^{\M\text{-app}}[rr] & & K_0 \ar^b[d] & M'_0 \ar^{\M\text{-app}}[d] \ar^j[dl] \ar@{.>}_{\exists \varepsilon}[dlll] \\
	I_{M_1} \ar[r] \ar[d] \ar_0[drr] & A \ar^l[dr] & M_0 \ar[l] \ar_a[d] & B \ar@/^1pc/@{.>}[ll] \ar^i[dl]\\
	\mho M_1 & & X &} \]
	
As $a$ is an $\M$-approximation, then there exists a morphism $j:M'_0 \to M_0$ which makes the lower-right square commute. Then, there exists a morphism $M'_1 \to K_0$ which makes it a morphism of short exact sequences.
	
As $M_1 \to K_0$ is an $\M$-approximation, then there exists $\delta : M'_1 \to M_1$ which makes the upper triangle commute. Since $I_{M_1}$ is injective and $M'_1 \to M'_0$ is an inflation, there exists a morphism $\varepsilon: M'_0 \to I_{M_1}$ which leads to a morphism of short exact sequences.
	
All the conditions are required to build a morphism $k:B \to A$ such that $l \circ k=i + 0$ since $I_{M_1} \to A \to X=0$. Then we have shown the result.
\end{proof}

\subsection{Weak equivalences and fibrations}

\begin{defi}
We recall that $G$ is the functor \[\begin{array}{ccccc}
G & : & \E & \to & \mathrm{Mod} \overline{\M} \\
 & & X & \mapsto & \overline{\E}(-,X)|\overline{\M} \\
\end{array}\]
We call by $\W$, the weak equivalences, the class of morphisms $f$ for which $Gf$ is an isomorphism.
\end{defi}

\begin{defi}
Let $f : X \to Y$ and $g : A \to B$ be two morphisms. We say that $f \square g$ when, for any commutative square
\[ \xymatrix{ A \ar[r] \ar[d] & X \ar[d] \\ B \ar[r] \ar@{.>}[ur] & Y} \] there exists a morphism $B \to X$ such that both triangles commute.
For a class of morphisms $\mathcal{A}$, we call by \[A^\square=\{ g,~\forall f \in \mathcal{A}, f \square g \} \text{ and } ^\square A=\{ f,~\forall g \in \mathcal{A}, f \square g \}.\]
\end{defi}

Let \[J= \{ f : 0 \to \mho M, M \in \M \}.\] The morphisms of $J^\square$ are called fibrations and compose the class $\Fib$.

The next lemma shows that the $\mathrm{pr}\M$-approximation constructed in lemma \ref{lem:prmapp} is actually a weak equivalence. This permits to take cofibrant replacements as we will see later.

\begin{lem}\label{lem:prmfib}
Assume that $\M$ contains the projective objects. Let $h:A \to X$ be the $\mathrm{pr}\M$-approximation constructed in Lemma \ref{lem:prmapp}. Then $h \in J^\square \cap \W$.
\end{lem}

\begin{proof}
We first show that $h \in J^\square$. Let $\mho M \to X$ be a morphism. Since $\mho \M \in \mathrm{pr}\M$, and $A \to X$ is a $\mathrm{pr}\M$-approximation, there automatically exists a lift as wanted.

Next, we have to show that the morphism $\overline{\E}(-,A)|\overline{\M} \to \overline{\E}(-,X)|\overline{\M}$ is an isomorphism. It is surjective, since if we take a morphism $M \to X$, since $M \in \mathrm{pr}\M$, there exists a lift wanted.

Then, if $a:M \to A$ is a morphism such that $h \circ a=0$ (we will see later the case where this morphism factorizes through an injective module). Using the same notations as in Lemma \ref{lem:prmapp}, as $M$ is rigid, there exists $\begin{pmatrix} b_1 \\ b_2 \end{pmatrix}:M \to I_1 \oplus M_0$ such that \[\pi \circ b_1 + r \circ b_2=a\] where $\pi:I_1 \to A$. As $\alpha:M_1 \to K_0$ is an $\M$-approximation, there exists $c:M \to M_1$ such that $b \circ \alpha \circ c=b_2$. Then we have \[a=\pi \circ b_1 + r \circ b_2\] so \[a=\pi \circ b_1 + r \circ b \circ \alpha \circ c.\] By the pushout of Lemma \ref{lem:prmapp}, $r \circ b \circ \alpha=\pi \circ \iota_1$ where $\iota_1:M_1 \to I_1$. Then \[a=\pi \circ b_1 + \pi \circ \iota_1 \circ c.\] This shows that $a$ factorizes through an injective module.

Finally, if we suppose that $h \circ a$ factorizes through an injective $J$, for example $h \circ a= \mu \circ \nu$, as $J \in \mathrm{pr}\M$, there exists $\tilde{a}$ such that $h \circ \tilde{a}=\mu$. We then proceed as above with the morphism $a-\tilde{a} \circ \nu$. This finishes to show the result.
\end{proof}

We recall that acyclic fibrations are those morphisms that are both fibrations and weak equivalences.

\begin{lem}\label{lem:relev} 
Let $f:X \to Y$ be an acyclic fibration. If $\alpha$ is a morphism from an element $M$ of $\M$ to $Y$, then there exists $\beta:M \to X$ such that $f \circ \beta=\alpha$.
\end{lem}

\begin{proof}
As $f$ is a weak equivalence, there exists $\tilde{\beta}:M \to X$, $\iota_M:M \to I_M$ and $\gamma:I_M \to Y$ such that $\alpha+\gamma \circ \iota_M=f \circ \tilde{\beta}$. As $f \in J^\square$ and $I_M \in \mho \M$, there exists $\delta:I_M \to X$ such that $f \circ \delta=\gamma$. Then we have \[\alpha=f \circ (\tilde{\beta}-\delta \circ \iota_M).\]
\end{proof}

\begin{lem}\label{lem:caracfibex}
Assume that $\M$ contains the projective objects. Let $f:X \to Y$ be an acyclic fibration. Then it is automatically a deflation.
\end{lem}

\begin{proof}
Let $P_Y$ be a projective cover of $Y$. As $P_Y \in \M$, and $f$ is an acyclic fibration, from the previous lemma, there exists a lift from $P_Y \to X$. From Buhler in \cite[Proposition 7.6, (ii)]{B}, as $P_Y \to Y$ is a deflation, then $f$ is a deflation.
\end{proof}

\subsection{Factorization}

Let us now show a characterization of the morphisms of $^\square (J^\square)$.

\begin{lem}\label{lem:jperpperp}
Suppose that $\mho M \to Y$ is a right $\mho \M$-approximation. A morphism $f : X \to Y$ is in $^\square (J^\square)$ if and only if it is a retract of the canonical injection $X \to X \oplus \mho M$.
\end{lem}

\begin{proof}
Let $f:X \to Y \in ^\square(J^\square)$. Let $\alpha:\mho M \to Y$ be a $\mho \M$-approximation of $Y$. Then, we have the following commutative square:

\[ \xymatrix{X \ar^{(^1_0)}[r] \ar_f[d] & X \oplus \mho M \ar^{(f~\alpha)}[d] \\ Y \ar@{=}[r] \ar@{.>}^s[ur] & Y}\]
The morphism $(f~\alpha)$ belongs to $J^\square$. Indeed, if $\mho M' \to Y$ is a morphism, as $\alpha$ is a $\mho \M$-approximation, there exists a lift as wanted (which is zero on $X$).
As $f \in ^\square(J^\square)$, there exists $s:Y \to X \oplus \mho M$ which makes both triangles commute. Then, $f$ is a retract of the canonical injection $X \to X \oplus \mho M$.

Conversely, it is well-known that the lifting property $\square$ is stable under retract.

%
%
\end{proof}

\begin{lem}\label{cor:fact}
Under the assumption that there exist some $\mho \M$-approximations, any morphism can be factorized through a morphism in $^\square (J^\square)$ followed by a morphism in $J^\square$.
\end{lem}

\begin{proof}
Let $f:X \to Y$ be a morphism. It factorizes through $X \to X \oplus \mho M \to Y$ by $\begin{pmatrix}1 \\ 0 \end{pmatrix}$ and $\begin{pmatrix}f & \alpha \end{pmatrix}$, where $\alpha$ is a $\mho \M$-approximation. The first morphism is in $^\square(J^\square)$. As $\alpha$ is a $\mho \M$-approximation, then $\begin{pmatrix}f & \alpha \end{pmatrix}$ satisfies the lifting property defining $J^\square$. Then $\begin{pmatrix}f & \alpha \end{pmatrix} \in J^\square$.
\end{proof}

\subsection{Cofibrant objects and homotopy}

\subsubsection{Cofibrant objects}

We call an object "cofibrant" objects if it lifts along all morphisms in $J^\square \cap \W$. Fibrant objects are defined dually.

In this subsection, we characterize fibrant and cofibrant objects.

\begin{lem}\label{lem:fib}
Any object is fibrant.
\end{lem}

\begin{proof}
For any $X \in \E$, the map $X \to 0$ is a fibration.
\end{proof}

\begin{lem}\label{lem:cof}
Suppose that the subcategory $\M$ contains the projective objects. Let $C \in \E$. Then $C$ is cofibrant if and only if $C \in \mathrm{pr}\M$.
\end{lem}

\begin{proof}
Let $C \in \mathrm{pr}\M$. We introduce $\xymatrix{0 \ar[r] & M_1 \ar^{h'}[r] & M_0 \ar^h[r] & C \ar[r] & 0}$. Let $f:X \to Y$ be an acyclic fibration and $b:C \to Y$. As $f \in \W$, $Gf$ is an isomorphism, and there exists from Lemma \ref{lem:relev} a morphism $a:M_0 \to X$ such that \[f \circ a = b \circ h.\] Since $\M$ contains all the projective objects, Lemma \ref{lem:caracfibex} shows that $f$ is a deflation. Let $k:K \to X$ be the kernel of $f$. We then have a morphism of short exact sequences:
\[ \xymatrix{M_1 \ar^c[r] \ar_{h'}[d] & K \ar^k[d] \\ M_0 \ar_h[d] \ar^a[r] & X \ar^f[d] \\ C \ar_b[r] & Y}\]
As $k \in \overline{\M}^\perp$, there exists $I$ an injective object, $\alpha:M_1 \to I$ and $\beta:I \to X$ such that \[k \circ c= \beta \circ \alpha.\]
As $h'$ is an inflation, there exists $\beta':M_0 \to I$ such that \[\beta' \circ h'=\alpha.\]
\[ \xymatrix{M_1 \ar^c[rr] \ar_{h'}[dd] \ar_\alpha[dr] & & K \ar^k[dd] \\ & I \ar_\beta[dr] & \\ M_0 \ar_h[dd] \ar^a[rr] & &  X \ar^f[dd] \\ & & \\ C \ar_\gamma[uurr] \ar_b[rr] & & Y}\]
So, \[h' \circ a =\beta \circ \beta' \circ h'\] and there exists $\gamma:C \to X$ such that \[\gamma \circ h= \beta \circ \beta' + a.\]
Then, \[f \circ \gamma \circ h=f \circ a + f \circ \beta \circ \beta'\] and \[(f \circ \gamma-b) \circ h=f \circ \beta \circ \beta'.\] Then, in $\mathrm{mod}\overline{\M}$, we have the good lifting. By Demonet and Liu in \cite{DL}, we have that \[\mathrm{mod}\overline{\M} \simeq \mathrm{pr}\M/\mho \M.\] As $C \in \mathrm{pr}\M$, there exists $M \in \M$ such that the morphism $b$ factorizes through $\mho M$, let us say \[\xymatrix{& \mho M \ar^\varepsilon[dr] & \\ C \ar_b[rr] \ar^\delta[ur] & & Y}\] such that $\varepsilon \circ \delta=b$.
As $f \in J^\square$, there exists $\iota:\mho M \to X$ such that $f \circ \iota=\varepsilon$. Then $\iota \circ \delta:C \to X$ is the good candidate to lift $b$, it means that \[f \circ \iota \circ \delta=\varepsilon \circ \delta=b.\]

On the other hand, let $X$ be a cofibrant object.

We can say that the $\mathrm{pr} \M$-approximation is an acyclic fibration from Lemma \ref{lem:prmfib}, then a retraction, and $X$ is a direct summand of $A$.

But, let us draw another proof. Let \[0 \to Y \to M_0 \to X \to 0\] be an $\M$-approximation of $X$. We have the following diagram:
\[ \xymatrix{
0 \ar[r] & Y \ar@{=}[d] \ar[r] & M_0 \ar[d] \ar[r] & X \ar[r] \ar@{.>}[d] & 0 \\
0 \ar[r] & Y \ar[r] & I_Y \ar[r] & \mho Y \ar[r] & 0} \]
There exists a morphism $X \to \mho Y$ which induces a morphism of short exact sequences. Let \[0 \to \mho M_Y \to \mho Y \to \mho^2 Z \to 0\] be an $\mho \M$-approximation of $\mho Y$. We can suppose that the morphism $\mho M_Y \to \mho Y$ is an inflation (we can add a copy of the injective envelope $I_{\mho M_Y}$ of $\mho M_Y$ if necessary to $\mho Y$). We have:
\[ \xymatrix{
0 \ar[r] & Y \ar@{=}[dd] \ar[r] & M_0 \ar[dd] \ar[r] & X \ar[r] \ar@{.>}[dd]  \ar@{.>}^{\exists \psi}[dr] & 0 & 0 \ar[dl] \\
& & & & \mho M_Y \ar_a[dl] & \\
0 \ar[r] & Y \ar[r] & I_Y \ar[r] & \mho Y \ar[r] \ar^b[dl] \ar[d] & 0 & \\
& & \mho^2 Z \ar[dl] & \mho M_0 \ar@{.>}^{\exists \varphi}[l] & & \\
& 0 & & & &} \]
We can add a copy of $\mho Y$ if necessary to have $a$ a weak equivalence.

As $X$ is cofibrant, there exists $\psi:X \to \mho M_Y$ such that $a \circ \psi=\alpha$. Then \[b \circ \alpha=b \circ a \circ \psi=0.\] Then there exists $\varphi: \mho M_0 \to \mho^2 Z$, which makes the triangle commute. However, from Wakamatsu's lemma, the object $\mho Z \in \overline{\M}^\perp$, so $\mho^2 Z \in \mho \overline{\M}^\perp$. Then $\varphi=0$ and $\mho Y \to \mho^2 Z=0$ and there is a section to the morphism $a$. Then $Y \in \M$ and $X \in \mathrm{pr}\M$.
\end{proof}

\subsubsection{Homotopy}

We recall here cylinder objects and left homotopies.

\begin{defi}
Let $X \in \E$. A cylinder object for $X$ is a factorization of the morphism $\nabla:X \oplus X \to X$ (which is the identity on each copy of $X$) through $X'$, where $X' \to X$ is a weak equivalence.

Let $f,g:X \to Y$ be two morphisms. A left homotopy from $f$ to $g$ is a morphism $h:X' \to Y$, where $X'$ is a cylinder object for $X$, such that $h \circ (\nabla_1 ~\nabla_2)=(f~g)$ where $(\nabla_1 ~\nabla_2)$ is the morphism $X \oplus X \to X'$ in the factorization of $\nabla$.
\end{defi}

Dually, we can define path objects and right homotopies.

\begin{defi}
Let $Y \in \E$. A path object for $Y$ is a factorization of the morphism $\Delta:Y \to Y \oplus Y$ (which is the identity on each copy of $Y$) through $Y'$, where $Y \to Y'$ is a weak equivalence.

Let $f,g:X \to Y$ be two morphisms. A right homotopy from $f$ to $g$ is a morphism $k:X \to Y'$, where $Y'$ is a path object for $Y$, such that $\begin{pmatrix}\Delta_1 \\ \Delta_2\end{pmatrix} \circ k=\begin{pmatrix}f\\g\end{pmatrix}$ where $\begin{pmatrix}\Delta_1 \\ \Delta_2\end{pmatrix}$ is the morphism $Y' \to Y$ in the factorization of $\Delta$.
\end{defi}

\begin{lem}\label{lem:htp}
For two morphisms $f$ and $g$ from an object $X$ to $Y$,$f$ and $g$ are homotopic if and only if $f-g$ factorizes through $\overline{\M}^\perp$ (meaning that there exists $X' \in \overline{\M}^\perp$ and $\alpha:X' \to Y$ and $\beta : X \to X'$ such that $f-g=\alpha \circ \beta$).
\end{lem}

\begin{proof}
We begin by noting this fact: in the next diagram, a factorization of $\Delta$ is a path object if and only if it is isomorphic to $\xymatrix{Y \ar^{\left(^1_0\right)}[r] & Y \oplus V \ar^{\left(^{1~1}_{c~d}\right)}[r] & Y \oplus Y}$ for a $V \in \overline{\M}^\perp$.

Indeed, let $\xymatrix{Y \ar^r[r] & Y' \ar[r] & Y \oplus Y}$ be a factorization of $\Delta$. As $r$ is a section and $\E$ is weakly idempotent complete, it is isomorphic to some $Y \to Y \oplus V$. Then $r \in \W$ if and only if $V \in \overline{\M}^\perp$.

Now, let us suppose that $f-g$ factorizes in the following way:
\[ \xymatrix{X \ar^{f-g}[r] \ar_{\alpha}[dr] & Y \\ & A \ar_{\beta}[u]}\] with $A \in \overline{\M}^\perp$. The object $Y \oplus A$ is a path object. Then we have
\[ \xymatrix{
X \ar^{\left(^g_\alpha\right)}[rr] \ar_{\left(^f_g\right)}[drr] & & Y \oplus A=Y' \ar_{\left(^{1~\beta}_{1~0}\right)}[d] & & Y \ar_{\left(^1_0\right)}[ll] \ar^{\Delta}[dll] \\
& & Y \oplus Y & & } \]
and the morphism $\begin{pmatrix} 1 & \beta \\ 1 & 0 \end{pmatrix}$ gives a homotopy between $f$ and $g$.

Conversely, if $f$ is homotopic to $g$, then there exists $k:X \to Y'$ such that \[\begin{pmatrix}\Delta_1 \\ \Delta_2 \end{pmatrix} \circ k= \begin{pmatrix} f \\ g \end{pmatrix}.\] Then, $f-g=(\Delta_1 - \Delta_2) \circ  k$. As $Y'=Y \oplus A$, and $A \in \overline{\M}^\perp$, this finishes to show the result.
\end{proof}

\begin{rmk}
The notions of left and right homotopy are the same for the cofibrant objects.
\end{rmk}

The following lemma is a corollary of the theorem of Demonet and Liu in \cite{DL}. However, we give a direct proof here.

\begin{lem}\label{lem:mho}
With the above notations, we have \[\mathrm{pr}\M \cap \overline{\M}^\perp = \mho \M.\]
\end{lem}

\begin{rmk}
This is true even when $\M$ is not contravariantly finite. In addition, the direct inclusion is true even if $\M$ is not rigid.
\end{rmk}

\begin{proof}
The indirect inclusion only uses the rigidity of $\M$. Indeed, let $\mho M \in \mho \M$. As there exists a short exact sequence $0 \to M \to I \to \mho M$, then $\mho \M \in \mathrm{pr}\M$. Moreover, for any $M' \in \M$, $\overline{\E}(M',\mho M)|\overline{\M}=0$ since $\M$ is rigid.
Now, let  $X \in \mathrm{pr}\M \cap \overline{\M}^\perp$. As $X \in \mathrm{pr}\M$, we have a short exact sequence \[0 \to M_1 \to M_0 \to X \to 0,\] with $M_1,M_0 \in \M$. We have the following diagram:
\[\xymatrix{ M_1 \ar@{=}[d] \ar[r] & M_0 \ar^c[r] \ar^\alpha[d] & X \ar^k[d] \\
M_1 \ar[r] & I \ar_\psi[d] \ar_h[r] \ar@{.>}^\beta[ur] & \mho M_1 \ar^\xi[d] \\
& \mho M_0 \ar@{=}[r] \ar@{.>}_\varphi[ur] & \mho M_0}\]
In order to show that $X \in \mho \M$, we will show that the short exact sequence \[0 \to X \to \mho M_1 \to \mho M_0 \to 0\] splits. As $X \in \overline{\M}^\perp$, the morphism $c:M_0 \to X$ factorizes through $I$. We call by $\alpha:M_0 \to I$. There exists $\beta : I \to X$ such that $\beta \circ \alpha = c$. Let $k:X \to \mho M_1$ and $h:I \to \mho M_1$. Then, \[k \circ \beta \circ \alpha=k \circ c = h \circ \alpha.\] Then, $h-k \circ \beta$ factorizes through $\mho M_0$. There exists $\varphi:\mho M_0 \to \mho M_1$ and, if we call by $\psi:I \to \mho M_0$, then, we have \[h=k \circ \beta+\varphi \circ \psi.\]
Let $\xi:\mho M_1 \to \mho M_0$ in the short exact sequence (see diagram for sake of clarity). Then we have \[\xi \circ \varphi \circ \psi = \xi \circ (h-k \circ \beta) = \xi \circ h - \xi \circ k \circ \beta = \psi - 0.\] As $\psi$ is a deflation, we can conclude that $\xi \circ \varphi=1$.
Then, there is a section to the short exact sequence and $X \in \mho \M$ (which is stable under direct summands since $\M$ is stable under direct summands.
\end{proof}

\begin{rmk}
We have the inclusion $\mathrm{pr}\M \subseteq \mathrm{copr} \mho \M$. See Lemma \ref{lem:copr}.
\end{rmk}

\subsection{Prefibration structures from rigid subcategories}\label{sec:pref}

In this section, we show that an exact category $\E$ is nearly equipped with a structure of a prefibration category in the sense of Anderson-Brown-Cisinski (for more details, see the book of Radulescu-Banu, \cite{R}).

We recall from Demonet and Liu in \cite{DL} that \[\begin{array}{ccccc}
G & : & \E & \to & \mathrm{Mod} \overline{\M} \\
 & & X & \mapsto & \overline{\E}(-,X)|\overline{\M} \\
\end{array}\] induces the following equivalence of categories \[ \mathrm{pr}\M/\mho \M \simeq \mathrm{mod}~\overline{\M}. \]

\begin{theo}\label{th:prefib}
Let $\E$ be a weakly idempotent complete exact category with enough injective and projective objects. Assume that $\M \subseteq \E$ is a rigid, contravariantly finite subcategory of $\E$ containing all the injective and projective objects, and stable under taking direct sums and summands. Suppose moreover that $\mho \M$ is contravariantly finite. Let \[J= \{ f : 0 \to \mho M, M \in \M \}\] and $J^\square$ be the class of fibrations. Let $\W$ be the class of morphisms whose images under the functor $G$ are isomorphisms.

Then $\E$ is almost equipped with the structure of a prefibration category, meaning that there exist two classes of morphisms, $\W$, the weak equivalences, and $\Fib$ the fibrations, such that:
\begin{enumerate}[label=(\roman*)]
\item The space $\W$ is stable under retracts and satisfies the two out of three axiom.
\item The space $\Fib$ is stable under composition, and all isomorphisms are fibrations.
\item Pullbacks exist along fibrations, and the pull-back of a fibration is a fibration. Moreover, if $\xymatrix{0 \ar[r] & A \ar^i[r] & B \ar^p[r] & Y \ar[r] & 0}$ is a short exact sequence, if $Gi$ is a monomorphism and $f:X \to Y$ is an acyclic fibration ($f \in \Fib \cap \W$), then $h$ defined by the following pull-back is an acyclic fibration.

\[\xymatrix{& E \ar_h[d] \ar^a[r] & X \ar^f[d] \\
A \ar_i[r] & B \ar_p[r] & Y}\]

\item There exist path objects : for any object $X$, the diagonal map $X \to X \oplus X$ can be factorized through $X'$, where the first morphism is a weak equivalence.

\item For any object $B$, the morphism $0 \to B$ is a fibration.
\end{enumerate}
\end{theo}

Let us first show the following lemma:

\begin{lem}
Let \[\xymatrix{& E \ar_h[d] \ar^a[r] & X \ar^f[d] \\
A \ar_i[r] & B \ar_p[r] & Y}\] where $(i,p)$ is a short exact sequence, $Gi$ a monomorphism and $f \in J^\square \cap \W$. Then $h \in J^\square \cap \W$.
\end{lem}

\begin{proof}
First, we have, without using the fact that $Gi$ is a monomorphism, that $h \in J^\square$. 

Second, let us show that $Gh$ is a monomorphism. As $h$ is a deflation (since it belongs to $J^\square$, and from lemma \ref{lem:caracfibex}), the morphisms $f$ and $h$ have the same kernel. Then we have the following diagram:

\[\xymatrix{& K \ar@{=}[r]\ar[d] & K \ar[d] \\ A \ar@{=}[d] \ar^k[r] & E \ar_h[d] \ar^a[r] & X \ar^f[d] \\ A \ar_i[r] & B \ar_p[r] & Y}\]

We have a short exact sequence $0 \to A \to E \to X \to 0$. Let us show that $Gh$ is a monomorphism. Let $\beta:D \to GE$ be such that $Gh \circ \beta=0$. We show that $\beta=0$. We have \[ Gp \circ Gh \circ \beta=0.\] Then \[G(p \circ h) \circ \beta=0\] and \[ G(f \circ a) \circ \beta=0.\] Thus, we have \[ Gf \circ Ga \circ \beta=0.\] As $Gf$ is a monomorphism since $f \in \W$, we have \[Ga \circ \beta=0.\] The fact that $G$ is left exact shows that there exists $c:D \to GA$ such that \begin{equation}
\label{eq:c}\beta=Gk \circ c.\end{equation} Moreover, by hypothesis \[Gi \circ c = Gh \circ \beta=0.\] As $Gi$ is a monomorphism, $c=0$ and from (\ref{eq:c}) we have $\beta=0$. This shows that $Gh$ is a monomorphism.

Now we show that $Gh$ is an epimorphism. Let $b:M \to B$ be a morphism, where $M \in \M$. Then we have $p \circ b:M \to Y$. From Lemma \ref{lem:relev}, there exists $\alpha:M \to X$ such that \[ f \circ \alpha=p \circ h.\] Then, from the pullback property, there exists $\varphi:M \to E$ such that $h \circ \varphi=b$ and this shows that $Gh$ is an epimorphism.
\end{proof}

\begin{proof}[Proof of the theorem]
\begin{enumerate}[label=(\roman*)]
\item The first item is well-known.
%
%
%
%
%
\item This follows from the fact that fibrations are defined by a lifting property.
\item Fibrations are deflations (this is because $\E$ is weakly idempotent complete, see lemma \ref{lem:caracfibex}), then pullbacks exist along fibrations. The rest of the item is the previous lemma.
\item We have the factorization from lemma \ref{cor:fact}. It means that any morphism can be factorised through a morphism of $^\square(J^\square)$ followed by a morphism of $J^\square$. Indeed, it only uses the fact that there exist some $\mho \M$-approximations, which we suppose in the hypotheses of this theorem. Then for any $X$, the diagonal $X \to X \oplus X$ can be factorised $X \to X' \to X \oplus X$, where the first morphism is a weak equivalence (we have seen that the morphisms of $^\square(J^\square)$ are weak equivalences from lemma $\ref{lem:jperpperp}$).
\item By definition of $J$, any object is fibrant, then $0 \to B$ is a fibration for any object $B$ of $\E$.
\end{enumerate}
\end{proof}

\subsection{Theorem of Quillen}

In this section, we show that the theorem of Quillen is satisfied. We recall that $\W$ is the class of morphisms whose images under $G$ are isomorphisms.

\begin{theo}
Let $\E$ be a weakly idempotent complete exact category with enough injective and projective objects. Assume that $\M \subseteq \E$ is a rigid, contravariantly finite subcategory of $\E$ containing all the injective and projective objects, and stable under taking direct sums and summands. Suppose moreover that $\mho \M$ is contravariantly finite.

Let $\mathrm{Ho}~\E$ be the localization of $\E$ at the class $\W$ of weak equivalences. Let $\mathrm{mod} ~\overline{\M}$ be the category of finitely presented $\overline{\M}$ modules. There is an equivalence of categories \[\mathrm{Ho}~\E \simeq \mathrm{mod}~ \overline{\M}. \]
\end{theo}

\begin{rmk}
In order to prove this theorem, we need the following lemmas. The proofs of these lemmas are well-known, but we give here some details in order to show that the restriction of (iii) in Theorem \ref{th:prefib} does not affect the results. Then, some parts of the proofs simplify due to the particular shape of the relation of homotopy.
\end{rmk}

\begin{lem}
Let $A,B,X$ be three cofibrant objects. Let $F$ be the functor $\E(-,X)/\sim$ where $\sim$ is the right homotopy relation. Then $F$ sends weak equivalences in $\E$ to isomorphisms.
\end{lem}

\begin{proof}
This functor is well-defined because the relation $\sim$ behaves well with the right composition. Let now $f:A \to B$. We factor $f$ through a morphism $g:A \to C \in ^\square(J^\square)$ followed by a morphism $p:C \to B \in J^\square$. As $g \in \W$ and $f \in \W$, then by the two-out-of-three property, $p$ is also a weak equivalence. In addition, $B$ is cofibrant, then there exists $w:B \to C$ such that $p \circ w=1$.

Then, we have \[1=F(1)=F(p \circ w)=F(p) \circ F(w).\] Then $F(p) \in \W$. This shows by the two-out-of-three property, that $F(f) \in \W$, and $Ff$ is surjective.

Let us now show that it is injective. If $\alpha:B \to X$ is such that $F\alpha=0$, then $\alpha \circ f \sim 0$. As $A$ is cofibrant, we have $\alpha \circ f \in (\mho \M)$ (from Lemma \ref{lem:htp}). This shows that $G(\alpha \circ f)=0$ since $\M$ is rigid. But $f \in \W$, then $Gf$ is an isomorphism, then $G\alpha=0$. As $B$ is cofibrant, $\alpha \in (\mho \M)$. Then $\alpha \sim 0$. Then $Ff$ is injective and this shows the lemma.
\end{proof}

Now we can show the following result:

\begin{lem}\label{lem:iso}
If $A$ and $B$ are two cofibrant objects, then \[\mathrm{Ho}~\E(A,B) \simeq \E(A,B)/\sim.\]
\end{lem}

\begin{proof}
Let $A$ and $B$ be two cofibrant objects.

\bigskip

Step 1: We show that, $f:A \to B$ is a weak equivalence if and only if it is a homotopy equivalence.

Suppose $f$ is a weak equivalence. We use the previous lemma with $X=A$. From the surjectivity of $Ff$, there exists $g:B \to A$ such that $g \circ f \sim 1$. Then, $f \circ g \circ f \sim f$. Now we apply the result to $X=B$ and have that \[Ff(f g)=Ff(1).\] However, $Ff$ is injective, then $f \circ g \sim 1$ and $f$ is a homotopy equivalence.

Suppose now that $f$ is a homotopy equivalence. Let $f'$ be a homotopy inverse for $f$. We have $f \circ f' \sim 1$ and $f' \circ f \sim 1$. Then $f \circ f'-1 \in (\mho \M)$ and $f' \circ f-1 \in (\mho \M)$. Then $G(f \circ f')=1$ and $G(f' \circ f)=1$. Then $Gf$ is an isomorphism and $f \in \W$.

This shows that $f=g$ in $\mathrm{Ho}\E(A,B)$ if and only if $f \sim g$.

\bigskip

Step 2: We now show the surjectivity of $\E(A,B) \to \mathrm{Ho}(A,B)$. We are going to use the book of Radulescu (see \cite[Theorem 6.4.2]{R}). Let us check the hypothesis with the pair of categories $(\E/\{f,Gf=0\},\underline{\W})$, where $\underline{\W}$ is the image of $\W$ in the quotient of $\E$ by $\{f, Gf=0\}$.

\begin{itemize}
\item The two out of three property is automatically checked.
\item If we have a pair of morphisms \[\xymatrix{& A' \ar_a[d] \\ B \ar_p[r] & A}\] such that $a \in \W$, then there exists $B',h:B' \to B$ and $k:B' \to A'$ such that the following square commutes:

\[\xymatrix{B' \ar^k[r] \ar_h[d] & A' \ar_a[d] \\ B \ar_p[r] & A}\]

Indeed, we introduce the factorization of $A' \to A$ by $b \in ^\square(J^\square)$ followed by $c \in J^\square$. As $a,b \in \underline{W}$, we also have $c \in \underline{W}$. Then $c$ is an acyclic fibration. Let $B'$ be a $\mathrm{pr}\M$-approximation of $B$. We lift the morphism $p \circ h$ to $c$, let us say $\begin{pmatrix}h_1 \\ h_2 \end{pmatrix}:B' \to A' \oplus \mho M$ (which is the shape of the factor in the factorization we have). Then $h_1$ is a lift from $B'$ to $A'$. We can check that the square commutes and $h \in \underline{W}$.
\item Suppose that we have \[\xymatrix{A \ar@<1ex>^f[r] \ar@<-1ex>_g[r] & B \ar^t[r] & B'}\] with $t \in \underline{W}$ and $t \circ f = t \circ g$. Then $Gf=Gg$, so, if $t':A' \to A \in \underline{W}$, we have $G(f \circ t')=G(g \circ t')$, then $f \circ t' - g \circ t' \in \overline{\M}^\perp$. Then $f \circ t'=g \circ t'$ in $\E/\{f,Gf=0\}$.
\end{itemize}
We can now apply Theorem 6.4.2 in \cite{R} and any morphism in $\mathrm{Ho}(\E(A,B))$ can be written as $f \circ s^{-1}$ with $s \in \underline{\W}$.
As $A$ and $B$ are cofibrant, we factor through $A' \to A$ in this way:
\[\xymatrix{& A' \ar^f[ddr] \ar_s[ddl] \ar[d] & \\ & A' \oplus \mho M \ar^{\alpha \in J^\square}[dl] \ar[dr] & \\ A \ar_a[rr] & & B}\] As $s \in \underline{\W}$, then $\alpha \in J^\square \cap \W$. Then we can lift $A' \oplus \mho M$ to $B$ by $(f~0)$.

From the theorem of Radulescu, if $\alpha \in \mathrm{Ho}(A,B)$, then there exists $s \in \W$ and $f \in \E(A,B)$ such that \[\alpha = \overline{f} \circ \overline{s}^{-1}.\] Then $\alpha=\overline{a}$ (since $\overline{a} \circ \overline{s}=\overline{f}$).

Then we have shown the surjectivity and then the lemma.
\end{proof}

\begin{proof}[Proof of the theorem of Quillen]
First of all, the functor is well-defined, since \[0 \to \mho M \in \W\] for any $M \in \M$ (because $\mathrm{Ext}^1(-,M)|\M=0$ implies $\overline{\E}(-,\mho M)|\M=0$).

Then, the functor is essentially surjective, since there exist some $\mathrm{pr}\M$-approximations, which are weak equivalences.

Next, from lemma \ref{lem:iso}, we have that $f=g$ in $\mathrm{Ho}\E(A,B)$ if and only if $f \sim g$ which immediately shows that the functor is faithful.

Finally, the functor is full, from the surjectivity of lemma \ref{lem:iso}.

The theorem of Demonet and Liu in \cite{DL} finishes to show the result.
\end{proof}

\subsection{The example of $A_7$}

In this subsection, we draw an example of the Theorem of case $A_7$. We are going to define a subcategory $\M$ which satisfies the hypotheses, and find $\mathrm{pr} \M$, $\mho\M$, the localization of $\E$, and remark that they are composed of the same objects.

Let $k$ be an algebraically closed field. Let $Q$ be a quiver of type $A_7$:
\[\xymatrix{1 & 2 \ar[l] & 3 \ar[l] & 4 \ar[l] & 5 \ar[l] & 6 \ar[l] & 7 \ar[l]}\]

Let $I$ be the ideal $\mathrm{rad}^2(kQ)$. Let $\E=\mathrm{mod}~{kQ/I}$ be the exact category we work with.

Let us draw the Auslander-Reiten quiver of $Q$:

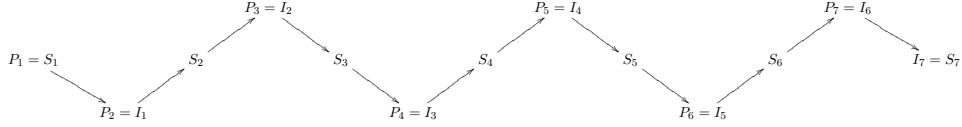
\[ \scalebox{0.5}{\xymatrix{
	& & & P_3=I_2 \ar[dr] & & & & P_5=I_4 \ar[dr] & & & & P_7=I_6 \ar[dr] & \\
	P_1=S_1 \ar[dr] & & S_2 \ar[ur] & & S_3 \ar[dr] & & S_4 \ar[ur] & & S_5 \ar[dr] & & S_6 \ar[ur] & & I_7=S_7 \\
	& P_2=I_1 \ar[ur] & & & & P_4=I_3 \ar[ur] & & & & P_6=I_5 \ar[ur] & & & \\
}}\]
Let $\M$ be the subcategory additively generated by all the projective and injective modules, plus $S_3$.

It means $\M$ is generated by the set $\{P_i,I_i,S_3, i \in \{1,\cdots,7\}\}$.

First, we are going to identify $\mathrm{pr} \M$, then $\mho\M$. Second, we identify the localization of $\E$ with respect to $\W$ and remark that it is the same as the quotient of $\mathrm{pr}\M$ by $\mho \M$.

\begin{itemize}
\item Note that all objects in $\M$ belong to $\mathrm{pr}\M$. Moreover, $S_2 \in \mathrm{pr}\M$.
	
Indeed, there is a short exact sequence $0 \to P_1 \to P_2 \to S_2 \to 0$. In the same way, we can show that $S_4 \in \mathrm{pr}\M$. Indeed, there exist a short exact sequence $0 \to S_3 \to P_4 \to S_4 \to 0$.
	
In addition, $S_5$ and $S_6$ do not belong to $\mathrm{pr}\M$.

Then, $\mathrm{pr}\M$ is generated by the objects $\{M \in \M, S_2,S_4\}$.
\item Let us now identify $\mho\M$. It is obvious that all the injective objects belong to $\mho\M$. Moreover, $S_4 \in \mho\M$.
	
Indeed, there is a short exact sequence $0 \to S_3 \to P_4=I_3 \to S_4 \to 0$.
	
In the same way, we can show that $S_2 \in \mho\M$. Indeed, there exist a short exact sequence $0 \to P_1 \to P_2=I_1 \to S_2 \to 0$.

Then, $\mho\M$ is generated by the injective objects plus $S_2$,$S_4$.

If we take the quotient of $\mathrm{pr} \M$ by $\mho\M$, we obtain a category of type $A_1 \times A_1$, generated by $P_1 \oplus S_3$. By Demonet and Liu, the quotient of $\mathrm{pr} \M$ by $\mho\M$ is equivalent to $\mathrm{mod}~\overline{\M}$. Then, $\mathrm{mod}~\overline{\M}$ is generated by $P_1 \oplus S_1$.

\item Finally, let us find the localization of $\E$ with respect to $\W$. For this, we have to inverse the quasi-equivalences. We recall that a morphism $f$ lies in $\W$ if and only if $\overline{\E}(-,f)|\overline{\M}$ is an isomorphism.

Is is easy to check that the morphism $f:P_5 \to S_5$ is a weak equivalence.

Indeed, let $g:M \to S_5$ be a nonzero morphism (where $M \in \M$ is indecomposable). Then, $M= P_5$ because, $f$ is the only nonzero morphism from an object of $\M$ to $S_5$. Then $\overline{\E}(-,f)|\overline{\M}$ is surjective.

Moreover, if there is a morphism $M \to P_5 \to S_5$ factorizing through an injective object, as $P_5$ is also injective, the injectivivty of $\overline{\E}(-,f)|\overline{\M}$ is automatic.

Then $f$ is a weak equivalence. In the localization, $f$ becomes an isomorphism from $P_5$ to $S_5$. As $P_5$ is also injective and, in the stable category, the injective objects become isomorphic to $0$, $S_5$ becomes isomorphic to $0$.

In a similar way, we can show that $S_6$ becomes isomorphic to $0$ in the localization. We also note that $S_2$ and $S_4$ become isomorphic to $0$ as they belong to $\mho \M$.

Finally, we show that $P_1$ and $S_3$ are not zero in the localization.

Indeed, let $g:P_3 \to S_3$ and $h:S_3 \to P_4$. We show that $\overline{\E}(-,h)|\overline{\M}$ is not injective. We have
\[\begin{array}{ccccc}
 & \overline{\E}(S_3,S_3)|\overline{\M} & \to & \overline{\E}(S_3,P_4)|\overline{\M} \\
 & \mathrm{Id}_{S_3} & \mapsto & h \\
\end{array}\]
The codomain turns to be isomorphic to $0$, and $h$ turns to $0$. However, $\mathrm{Id}_{S_3}$ is not $0$. Then the application is not injective.

In a similar way, we can show that $\overline{\E}(-,g)|\overline{\M}$ is not injective.

We show the same result for $P_1$ and the localization of $\E$ with respect to $\W$ is also generated by $P_1 \oplus S_3$.

\end{itemize}

\newpage

\section{The particular case of Frobenius categories}
We recall the notations:

Let $\mathcal{E}$ be a Frobenius category. Let $\M$ be a strictly full rigid subcategory which contains the injective objects, stable under taking sums and summands, and contravariantly finite. Let
\[ \mathrm{pr} \M = \{ X \in \E, \exists M_1,M_0 \in \M, 0 \to M_1 \to M_0 \to X \to 0 \} \] and
\[ \mho\M = \{ X \in \E, \exists M \in \M, I \in \mathrm{Inj}, 0 \to M \to I \to X \to 0 \}. \]

Let \[\begin{array}{ccccc}
G & : & \E & \to & \mathrm{Mod} \overline{\M} \\
 & & X & \mapsto & \overline{\E}(-,X)/\overline{\M} \\
\end{array}\] which induces the following equivalence of categories \[ \mathrm{pr}\M/\mho \M \simeq \mathrm{mod}~\overline{\M}. \]

For more details, see the article of Demonet and Liu, \cite{DL}.

As the injective objects are also projective, note that $\M$ contains the projective objects.

\subsection{A deeper study of $\mathrm{pr}\M$}

In this section, we show that the subcategory $\mathrm{pr}\M$ is better behaved when $\E$ is Frobenius.

\begin{lem}\label{lem:approx}
Let $X \in \E$. If there is an $\M$-approximation of $\Omega X$, then there exists a $\mho \M$-approximation of $X$. In particular, if $\M$ is contravariantly finite, then so is $\mho \M$.
\end{lem}

\begin{proof}
	
Let $a:M \to \Omega X$ be an $\M$-approximation of $\Omega X$. We can introduce $Y$ as the push-out of the pair of morphisms:
	
\[ \xymatrix{
	M \ar_{\iota_M}[d] \ar^a[r] & \ \Omega X \ar^c[d] \\
	I_M \ar_b[r] & Y} \]

in particular \begin{equation} \label{eq:mho0} c \circ a= b \circ \iota_M. \end{equation}
	
Let \[\xymatrix{0 \ar[r] & \Omega X \ar^{\iota_X}[r] & P_X \ar^{\pi_X}[r] & X \ar[r] & 0}\] be a short exact sequence expressing the fact that $\Omega X$ is a sizygy of $X$. We then have the following diagram:
	
\[ \scalebox{0.6}{\xymatrix{
	M \ar_{\iota_M}[dd] \ar^a[r] & \ \Omega X \ar^c[dd] \ar@{=}[dr] & \\
	& & \Omega X \ar^{\iota_X}[dd] \\
	I_M \ar_b[r] \ar_{\pi_M}[dd] & Y \ar^e[dd] \ar@{.>}^{\alpha}[dr] & \\
	& & P_X \ar^{\pi_X}[dd] \\
	\mho M \ar@{=}[r] & \mho M \ar@{.>}^{\beta}[dr] & \\
	& & X}} \]
	
Note that we have \begin{equation} \label{eq:mho01} e \circ b = \pi_M. \end{equation}
	
Since $\E$ is Frobenius, $P_X$ is injective, and there exists $\alpha: Y \to P_X$ such that \begin{equation}\label{eq:mho1} \iota_X=\alpha \circ c. \end{equation}
	
Then there exists $\beta$ completing the commutative square of the previous equation \ref{eq:mho1} to a morphism of short exact sequences: \begin{equation}\label{eq:mho2} \pi_X \circ \alpha = \beta \circ e \end{equation}
	
Then we have a short exact sequence \[0 \to Y \to P_X \oplus \mho M \to X \to 0.\] We now show that $(\pi_X~\beta)$ is a $\mho \M$-approximation (we note that, $\E$ being Frobenius, $\mho \M$ contains the projective modules).
	
Let $\gamma:\mho N \to X$ be a morphism. As $I_N$ is also projective, there exists $\delta:I_N \to P_X$ such that \begin{equation}\label{eq:mho3} \pi_X \circ \delta = \gamma \circ \pi_N \end{equation}
	
\[ \scalebox{0.8}{\xymatrix{
	M \ar_{\iota_M}[dd] \ar^a[r] & \ \Omega X \ar^c[dd] \ar@{=}[dr] & & \\
	& & \Omega X \ar^{\iota_X}[dd] & N \ar@{.>}_{\varepsilon}[l] \ar@{.>}@/^-3pc/_\zeta[ulll] \ar^{\iota_N}[dd] \\
	I_M \ar_b[r] \ar_{\pi_M}[dd] & Y \ar^e[dd] \ar^{\alpha}[dr] & & \\
	& & P_X \ar^{\pi_X}[dd] & I_N \ar@{.>}_{\delta}[l] \ar@{.>}@/^-3pc/_\eta[ulll] \ar^{\pi_N}[dd]\\
	\mho M \ar@{=}[r] & \mho M \ar^{\beta}[dr] & & \\
	& & X & \mho N \ar^{\gamma}[l] \ar@{.>}@/^-3pc/_\theta[ulll] \ar@{.>}_\kappa[uul]}} \]
	
We complete \ref{eq:mho3} into a morphism of short exact sequences by $\varepsilon$, meaning that \begin{equation} \label{eq:mho4} \iota_X \circ \varepsilon = \delta \circ \iota_N. \end{equation}
	
As $a:M \to \Omega X$ is an $\M$-approximation, there exists $\zeta:N \to M$ such that \begin{equation} \label{eq:mho5} a \circ \zeta = \varepsilon. \end{equation}
	
As $I_M$ is injective, and $\iota_N$ is an inflation, there exists $\eta: I_N \to I_M$ such that \begin{equation} \label{eq:mho6} \eta \circ \iota_N=\iota_M \circ \zeta. \end{equation}
Then there exists $\theta:\mho N \to \mho M$ such that: \begin{equation} \label{eq:mho7} \theta \circ \pi_N = \pi_M \circ \eta. \end{equation}
	
Then we have:
\begin{equation} \begin{aligned} \delta \circ \iota_N & = \iota_X \circ \varepsilon \text{ by (\ref{eq:mho4})} \\
& = \iota_X \circ a \circ \zeta \text{ by (\ref{eq:mho5})}\\
& = \alpha \circ c \circ a \circ \zeta \text{ by (\ref{eq:mho1})} \\
& = \alpha \circ b \circ \iota_M \circ \zeta \text{ by (\ref{eq:mho0}}) \\
& = \alpha \circ b \circ \eta \circ \iota_N \text{ by (\ref{eq:mho6}).}
\end{aligned} \end{equation}
	
Then, the morphism $\alpha \circ b \circ \eta - \delta$ factorizes through $\pi_N$.
	
Thus, there exists $\kappa: \mho N \to P_X$ such that \begin{equation} \label{eq:mho8} \kappa \circ \pi_N = \delta -\alpha \circ b \circ \eta. \end{equation}
	
Let us finally show that the morphism $\begin{pmatrix} \kappa \\ \theta \end{pmatrix}$ satisfies $(\pi_X~\beta) \begin{pmatrix} \kappa \\ \theta \end{pmatrix} = \gamma$.
	
As $\pi_N$ is a deflation, is suffices to show that \[\pi_X \circ \kappa \circ \pi_N + \beta \circ \theta \circ \pi_N= \gamma \circ \pi_N. \]
We have:
\begin{equation} \begin{aligned}
\pi_X \circ \kappa \circ \pi_N & =\pi_X(\delta-\alpha \circ b \circ \eta) \text{ by (\ref{eq:mho8})} \\
& = \gamma \circ \pi_N - \pi_X \circ \alpha \circ b \circ \eta \text{ by (\ref{eq:mho3})} \\
& = \gamma \circ \pi_N - \beta  \circ e \circ b \circ \eta \text{ by (\ref{eq:mho2})} \\
& = \gamma \circ \pi_N - \beta  \circ \pi_M \circ \eta \text{ by (\ref{eq:mho01})} \\
& = \gamma \circ \pi_N-\beta \circ \theta \circ \pi_N \text{ by (\ref{eq:mho7})}.
\end{aligned} \end{equation}
\end{proof}

\begin{defi}
We call by $\mathrm{copr}~\mho \M$ the class of objects $X \in \E$ for which there exist $M,M' \in \M$ and a short exact sequence $0 \to X \to \mho M \to \mho M' \to 0$.
\end{defi}

\begin{lem}\label{lem:copr}
We have the following equality: \[\mathrm{copr}~\mho \M = \mathrm{pr}\M.\]
\end{lem}

\begin{proof}
Let $C \in \mathrm{copr}~\mho \M$. We have the diagram:
\[ \xymatrix{M_1 \ar[r] \ar@{=}[d] & \Omega \mho M_0 \ar[r] \ar[d] & C \ar[d] \\ M_1 \ar[r]  & I \ar[r] \ar[d]& \mho M_1 \ar[d] \\ & \mho M_0 \ar@{=}[r] & \mho M_0} \]
We have two morphisms of short exact sequence \[ \xymatrix{ M \ar[r] \ar@<2pt>[d] & I \ar[r] \ar@<2pt>[d] & \mho M \ar@{=}[d] \\ \Omega \mho M \ar@<2pt>[u] \ar[r] & P \ar@<2pt>[u] \ar[r] & \mho M}\]
showing that $\Omega \mho M \simeq M \oplus I \in \M$. The other inclusion is proved in part 1, lemma \ref{lem:mho}.
\end{proof}

\begin{lem}\label{lem:mhomrigide}
The class $\mho \M$ is rigid.
\end{lem}

\begin{proof}
Let $M,M' \in \M$. We show that any morphism $f:\mho M \to \mho^2 M'$ is zero in $\overline{\E}$. Let us introduce the following diagram:
\[ \xymatrix{M \ar^h[r] \ar_{\iota_M}[d] & \mho M' \ar^{\iota_{\mho M'}}[d] \\
I_M \ar^g[r] \ar_{\pi_M}[d] \ar@{.>}^\alpha[ur] & I_{\mho M'} \ar^{\pi_{\mho M'}}[d] \\
\mho M \ar_f[r] \ar@{.>}_\beta[ur] & \mho^2 M'}\]
As $I_M$ is also projective and $\pi_{\mho M'}$ is a deflation, there exists $g$ which makes the lower square commute. Then, there exists $h$ such that the upper square commutes. Now, $h$ factorizes through an injective module since $\M$ is rigid. Thus there exists $\alpha:I_M \to \mho M'$ such that $\alpha \circ \iota_M=h$. Then, there exists $\beta:\mho M \to I_{\mho M'}$ such that \[ g=\iota_{\mho M'} \circ \alpha + \beta \circ \pi_M.\]
Then,
\[f \circ \pi_M=\pi_{\mho M'} \circ g = \pi_{\mho M'} \circ(\iota_{\mho M'} \circ \alpha + \beta \circ \pi_M)=\pi_{\mho M'} \circ \beta \circ \pi_M\]
As $\pi_M$ is a deflation, this shows that $f=\pi_{\mho M'} \circ \beta$ and thus factorizes through an injective. Then $\mho \M$ is rigid.
\end{proof}

\begin{lem}\label{lem:pr}
Let $A,C$ be objects of $\E$. Let $M \in \M$. Suppose that we have a short exact sequence \[0 \to A \to C \to \mho M \to 0.\] Then $C \in \mathrm{pr}\M$ if and only if $A \in \mathrm{pr}\M$.
\end{lem}

\begin{proof}
Using Lemma \ref{lem:copr}, we will show that $C \in \mathrm{copr}~\mho \M$ if and only if $A \in \mathrm{copr}~\mho \M$. First, suppose that $A \in \mathrm{copr}~\mho \M$. Then, we have the following diagram:
\[ \xymatrix{A \ar[r] \ar[d] & \mho M_1 \ar[r] \ar[d] & \mho M_0 \ar@{=}[d] \\ C \ar[r] \ar[d] & E \ar[r] \ar[d] & \mho M_0 \\ \mho M_2 \ar@{=}[r] & \mho M_2 &}\]
Where $E$ is the pushout of the morphisms $A \to \mho M_1$ and $A \to C$. By lemma \ref{lem:mhomrigide}, there is a section to the short exact sequence \[0 \to \mho M_1 \to E \to \mho M_2 \to 0,\] and $E \simeq \mho M_1\oplus \mho M_2$. Let $M=M_1 \oplus M_2$. Then we have a short exact sequence $0 \to C \to \mho M \to \mho M_0 \to 0$.

Conversely, if $C \in \mathrm{copr}~\mho \M$. Then, there exists a short exact sequence \[0 \to C \to \mho M^0 \to \mho M^1 \to 0.\] Let us take the push-out:

\[ \xymatrix{A \ar[r] \ar@{=}[d] & C \ar[r] \ar[d] & \mho M \ar[d] \\ A \ar[r] & \mho M^0 \ar[r] \ar[d] & E \ar[d] \\ & \mho M^1 \ar@{=}[r] & \mho M^1}\]
As $\mho \M$ is rigid from the previous lemma, then $E \in \mho \M$ showing that $A \in \mathrm{copr}\mho \M$.
\end{proof}

\begin{rmk}
A dual version of this lemma holds: If there is a short exact sequence $0 \to M \to C \to A \to 0$, then $A \in \mathrm{pr}\M$ if and only if $C \in \mathrm{pr}\M$.
\end{rmk}

\subsection{Weak equivalences}\label{sec:not}

In this section, we give a characterization of the weak equivalences that will be useful.

\begin{defi}
Let $\mathcal{A}$ be a subcategory of $\E$. Let $g:Y \to Z$ be a morphism. We say that $g \in \overline{\mathcal{A}}^\perp$ if, for any morphism $b: A \to Y$, with $A \in \mathcal{A}$, the morphism $g \circ b$ factorizes through an injective module.
\end{defi}

\begin{lem}\label{lem:carac}
Let $X,Y \in \E$. Let $f : X \to Y$. Then $Gf$ is an isomorphism if and only if, for any pushout
\[\xymatrix{X \ar^f[r] \ar_{\iota_X}[d] & Y \ar^g[d] \\
I_X \ar^u[r] \ar_{\pi_X}[d] & Z \ar[d] \\
\mho X \ar@{=}[r] & \mho X}\]
we have $\overline{g} : Y \to Z \in \overline{\M}^\perp$ and, for any pullback
\[\xymatrix{X \ar^f[r] & Y \\
Z \ar[r] \ar_{\tilde{g}}[u] & P_Y \ar[u] \\
\mho Y \ar@{=}[r] \ar[u] & \mho Y \ar[u]}\]
we have $\overline{\tilde{g}}: Z \to X \in \overline{\M}^\perp$.
\end{lem}

\begin{defi}
For a morphism $f:X \to Y$, we call a cone of $f$, any morphism $g$, obtained from a push-out of an inflation $X \to I$ for some injective $I$ along $f$. Then, we notice that $g$ is unique up to isomorphism in $\overline{\E}$.
\end{defi}

\begin{proof}
If $Gf$ is an isomorphism. Any diagram
\[\xymatrix{X \ar^f[r] \ar_{\iota_X}[d] & Y \ar^g[d] \\
I_X \ar^u[r] \ar_{\pi_X}[d] & Z \ar[d] \\
\mho X \ar@{=}[r] & \mho X}\]
induces a short exact sequence \[\xymatrix{0 \ar[r] & X \ar[r] & I_X \oplus Y \ar^{~~~(u~g)}[r] & Z \ar[r] & 0}.\] Let us show that $(\overline{u~g}) \in \overline{\M}^\perp$. Since $u : I_X \to Z$, with $I_X$ injective we already have that $\overline{\mathcal{E}}(\M,u)=0$. Let
\[\mathcal{E}(\M,g) : \mathcal{E}(\M,Y) \to \mathcal{E}(\M,Z) \] be the morphism which, with
$h : M \in \M \to Y$ associates the morphism $g \circ h$.

As $Gf$ is an isomorphism, there exists $h' \in \mathcal{E}(\M,X)$ such that $f \circ \overline{h'}=\overline{h}$. Then \[\overline{h'} \to f \circ \overline{h'}=\overline{h} \to g \circ f \circ \overline{h'}=g \circ h=0,\] and $\overline{\mathcal{E}}(\M,g)=0$.

The other case is similar and left to the reader.
%

Conversely, let us suppose that $\xymatrix{I_X \oplus Y \ar^{\overline{(\tilde{u}~\tilde{g})}}[r] & Z \in \overline{\M}^\perp}$ and $\xymatrix{Z \ar[r] & X \oplus P_Y \in \overline{\M}^\perp}$. Let us introduce the short exact sequences
\begin{equation}\label{eq:1}
\xymatrix{0 \ar[r] & X \ar[r] & I_X \oplus Y \ar^{(\tilde{u}~\tilde{g})}[r] & Z \ar[r] & 0}
\end{equation}
and
\begin{equation}\label{eq:2}
\xymatrix{0 \ar[r] & Z \ar[r] & 
X \oplus P_Y \ar[r] & Y \ar[r] & 0}
\end{equation}

Apply the functor $G=\mathcal{E}(A,-)$ where $A \in \M$ to the short exact sequence (\ref{eq:1}). As $G$ is left exact, we have the following long exact sequence:

\[ 0 \to \mathcal{E}(A,X) \to \mathcal{E}(A,Y \oplus I_A) \to \mathcal{E}(A,Z) \to \mathrm{Ext}^{1}(A,X) \to \cdots \]

If we factorize through the injective modules, then we get:

\[ \cdots \to \overline{\E}(A,\Omega Z) \to \overline{\mathcal{E}}(A,X) \to \overline{\mathcal{E}}(A,Y) \to \overline{\mathcal{E}}(A,Z) \to \overline{\E}(A,\mho X) \to \cdots \]

Moreover, we have $\overline{\mathcal{E}}(A,g)=0$ because $A \in \M$ and $g \in \overline{\M}^\perp$. Then $Gf$ is an epimorphism.

Now let us apply the functor $G$ to the short exact sequence (\ref{eq:2}) and factorize through projective modules. We then have:

\[ \cdots \to \underline{\E}(A,\Omega Y) \to \underline{\mathcal{E}}(A,Z) \to \underline{\mathcal{E}}(A,X) \to \underline{\mathcal{E}}(A,Y) \to \mathrm{Ext}^{1}(A,Z) \to \cdots \]

We have $\underline{\mathcal{E}}(A,\tilde{g})=0$ because $A \in \M$ and $\tilde{g} \in \overline{\M}^\perp=\underline{\M}^\perp$. Then $Gf$ is a monomorphism, thus an isomorphism.
\end{proof}

\subsection{Fibrations and cofibrations}

Let us recall the notations:

Let $\mathcal{W}$ be the class of morphisms $f$ such that $Gf$ is an isomorphism. Such morphisms are called weak equivalences. Let \[J= \{ f : 0 \to \mho M, M \in \M \}\] and \[I=\{ f : M_0 \to X \oplus I_0, X \in \mathrm{pr}\M \} \cup \{0 \to M,M \in \M \}\] where $0 \to M_1 \to M_0 \to X \to 0$ is a short exact sequence and $I_0$ appears in the short exact sequence $0 \to M_0 \to I_0 \to \mho M_0$.

The morphisms of $J^\square$ are called fibrations and compose the class $\Fib$. The following lemma characterizes fibrations. We recall that $g \in \mho\overline{\M}^\perp$ if, for any morphism $b:\mho M \to Y$, with $M \in \M$, the morphism $g \circ b$ factorizes through an injective module.

\begin{lem}\label{lem:caracfib}
Suppose that $\E$ is a weakly idempotent complete category. Let $f:X \to Y$ be a morphism in $\E$. Then $f$ is a fibration if and only if $f$ is a deflation, and $g \in \mho\overline{\M}^\perp$, where $g$ is a cone of $f$, defined by
\[\xymatrix{X \ar^f[r] \ar_{\iota_X}[d] & Y \ar^g[d] \\
I_X \ar^u[r] \ar_{\pi_M}[d] & Z \ar[d] \\
\mho X \ar@{=}[r] & \mho X.}\]
\end{lem}

\begin{proof}
Let $f:X \to Y$ be a deflation. Let $g$ be a cone of $f$. Suppose that $g \in \mho\overline{\M}^\perp$. Let us show that $f \in J^\square$.
Let $b:\mho M \to Y$, where $\mho M \in \mho \M$. As $(g~u) \in \mho\overline{\M}^\perp$, then the morphism $(g~u) \circ \begin{pmatrix}b \\ 0 \end{pmatrix}$ factorizes through an injective. There exists $\iota:\mho M \to I$, $\alpha_1 : I \to Y$ and $\alpha_2: I \to I_x$ such that we have \[ \begin{pmatrix} g & u \end{pmatrix} \begin{pmatrix} b-\alpha_1 \circ \iota \\ 0 - \alpha_2 \circ \iota \end{pmatrix} = 0. \] We have the following diagram:

\[ \xymatrix{
& & X \ar^{\left(^f_{\iota_X}\right)}[d] \\
\mho M \ar@{.>}^\beta[urr] \ar^{\left(^b_0\right)}[rr] \ar_{\iota}[dr] & & Y \oplus I_X \ar^{(g~u)}[dd] \\
& I \ar@{.>}^\gamma[uur] \ar@{.>}_{\left(^{\alpha_1}_{\alpha_2}\right)}[ur] & \\
& & Z.} \]

Then, there exists $\beta:\mho M \to X$ such that \[ \begin{pmatrix} b-\alpha_1 \circ \iota \\ \alpha_2 \circ \iota \end{pmatrix} = \beta \circ \begin{pmatrix} f \\ \iota_X \end{pmatrix}. \]

As $I$ is also projective since $\E$ is a Frobenius category, and $f$ is a deflation, there exists $\tilde{\alpha_1}:I \to X$ such that \[ \alpha_1 = f \circ \tilde{\alpha_1}. \]
Then, we have \[f \circ (\beta - \tilde{\alpha_1} \circ \iota)=b\] and this shows the first implication.

On the other hand, let $g$ be a cone of $f:X \to Y$. If $f \in J^\square$, then it is a deflation. Indeed, let $\pi_Y:P_Y \to Y$ be a projective pre-cover of $Y$ (we can equivalently introduce a $\mho \M$-approximation of $Y$). There exists a morphism $\varphi : P_Y \to X$ which makes the diagram commutative. As $f \circ \varphi = \pi_Y$ is a deflation and $\E$ is a weakly idempotent complete category, then $f$ is also a deflation. Let us now show that $g \in \mho\overline{\M}^\perp$. Let $h : B \to Y$, where $B \in \mho\M$. We have to show that $\overline{g \circ h}=0$. As $f \in J^\square$, there exists $\varphi : B \to X$ such that $f \circ \varphi=h$. Then $g \circ h= g \circ f \circ \varphi$. As we know $g \circ f$ factorize through an injective, so does $g \circ h$. Then $g \in \mho\overline{\M}^\perp$.
\end{proof}

\subsection{Factorizations}

We start by an important Lemma.

\begin{lem}\label{lem:ega}
Let $\E$ be a weakly idempotent complete Frobenius category. Assume that $\M \subseteq \E$ is a rigid, contravariantly finite subcategory of $\E$ containing all the injective objects, and stable under taking direct sums and summands.

Let \[J= \{ f : 0 \to \mho M, M \in \M \}\] and \[I=\{ f : M_0 \to X \oplus I_0, X \in \mathrm{pr}\M \} \cup \{0 \to M,M \in \M \}.\]
Then
\[I^\square = J^\square \cap \W.\]
\end{lem}

\begin{proof}
First, let us show that $J^\square \cap \W \subseteq I^\square$. Let $f \in J^\square \cap \W$. Let
\[\xymatrix{M_0 \ar^a[r] \ar_{\left(^{h}_{\iota_0}\right)}[d] & X \ar^f[d] \\
A \oplus I_0 \ar_{(b_1~b_2)}[r] & Y} \] be a commutative square where $\begin{pmatrix} h \\ \iota_0 \end{pmatrix} \in I$, meaning that $\iota_0:M_0 \to I_0$ is an injective cover, $A \in \mathrm{pr}\M$, and $M_1 \to M_0 \to A$ is a short exact sequence, with $M_1,M_0 \in \M$.
Then we have \begin{equation}
\label{eq:J1}b_1 \circ h + b_2 \circ \iota_0=f \circ a.\end{equation}
As $f$ is a deflation, by Lemma \ref{lem:caracfibex} there exists $K$ such that $0 \to K \to X \to Y \to 0$ is a short exact sequence. Since $I_0$ is projective, the pullback of $I_0 \to Y$splits and we have a short exact sequence \[0 \to K \oplus I_0 \to I_0 \to Y \to 0.\]

From equation (\ref{eq:J1}), we have that \[b_1 \circ h=(f~-b_2) \circ \begin{pmatrix}a \\ \iota_0 \end{pmatrix}.\]
Then, there exists $\begin{pmatrix}c \\ \mu \end{pmatrix}$ such that
\[\xymatrix{
0\ar[d] & 0\ar[d] \\
M_1 \ar_k[d] \ar^{\left(^c_\mu\right)}[r] & K \oplus I_0 \ar^{\left(^{\tilde{g}~\alpha}_{0~1}\right)}[d] \\
M_0 \ar_h[d] \ar^{\left(^a_{\iota_0}\right)}[r] & X \oplus I_0 \ar^{(f~-b_2)}[d] \\
A \ar[d] \ar^{b_1}[r] & Y \ar[d] \\
0 & 0}\] is a morphism of short exact sequences.

As $\tilde{g} \in \overline{\M}^\perp$, the morphism $M_1 \to K \to X$ factorizes through an injective. We can suppose without loss of generality, that it factorizes through $I_0$. Then we have:
\[\xymatrix{
M_1 \ar_k[dd] \ar^{\left(^c_\mu\right)}[rr] \ar^{\alpha_A}[dr] & & K \oplus I_0 \ar^{\left(^{\tilde{g}~\alpha}_{0~1}\right)}[dd] \\
& I_0 \ar^{\left(^{\beta_1}_{\beta_2}\right)}[dr] & \\
M_0 \ar@{.>}^{\iota_0}[ur]\ar^{\left(^a_{\iota_0}\right)}[rr] & & X \oplus I_0 }\] with \[\begin{pmatrix} \beta_1 \\ \beta_2 \end{pmatrix} \circ \alpha_A = \begin{pmatrix} a \circ k \\ \iota_0 \circ k \end{pmatrix}.\]
As $I_0$ is injective, and $k$ is an inflation, we can lift $\alpha_A$ to $\iota_0 : M_0 \to I_0$ such that $\iota_0 \circ k=\alpha_A$. Then,
\[\xymatrix{M_0 \ar_h[dd] \ar^{\left(^a_{\iota_0}\right)}[rr] \ar^{\iota_0}[dr] & & X \oplus I_0 \ar^{(f~-b_2)}[dd] \\
& I_0 \ar_{\left(^{\beta_1}_{\beta_2}\right)}[ur] & \\
A \ar^{b_1}[rr] & & Y }.\]
Then, there exists $\begin{pmatrix} \alpha_1 \\ \alpha_2 \end{pmatrix}:A \to X \oplus I_0$ such that \begin{equation}\label{eq:J2} \begin{pmatrix} a \\ \iota_0 \end{pmatrix} = \begin{pmatrix} \beta_1 \\ \beta_2 \end{pmatrix} \circ \iota_0 + \begin{pmatrix} \alpha_1 \\ \alpha_2 \end{pmatrix} \circ h. \end{equation}
Then, we have a morphism $A \oplus I_0 \to X$ which makes the upper triangle commute:
\[\xymatrix{M_0 \ar^a[r] \ar_{\left(^{h}_{\iota_0}\right)}[d] & X \ar^f[d] \\
A \oplus I_0 \ar_{(b_1~b_2)}[r] \ar^{(\alpha_1 ~ \beta_1)}[ur] & Y}. \]
We have \[(\alpha_1~\beta_1) \circ \begin{pmatrix} h \\ \iota_0 \end{pmatrix}=\alpha_1 \circ h + \beta_1 \circ \iota_0=a\] from equation \ref{eq:J2}.

The morphism $M_0 \to A \oplus I_0$ is an inflation. Let us introduce its cokernel $C$. Then we have the push-out:
\[ \xymatrix{
M_0 \ar^{h}[r] \ar[d] & A \ar^{\gamma_1}[d] \ar@/^1pc/^{f \circ \alpha_1 - b_1}[ddr] & \\
I_0 \ar_{\gamma_2}[r] \ar@/_1pc/_{f \circ \beta_1-b_2}[drr] & C \ar^{\exists \psi}[dr] & \\
& & Y} \]
Then there exists a unique $\psi:C \to Y$ such that:
\begin{equation}\label{eq:J3} \psi \circ \gamma_2=f \circ \beta_1-b_2 \end{equation}
and
\begin{equation}\label{eq:J4} \psi \circ \gamma_1=f \circ \alpha_1-b_1 \end{equation}
The push-out $C$ is exactly the cokernel of the morphism $M_1 \to M_0 \to I_0$. Then $C \simeq \mho M_1$. As $f \in J^\square$, if we still denote by $\psi$ the morphism from $\mho M_1$ to $Y$, then there exists $\zeta:\mho M_1 \to X$ such that $f \circ \zeta = \psi$. Then the equations \ref{eq:J3} and \ref{eq:J4} give respectively \[f \circ \zeta \circ \gamma_1= f \circ \alpha_1 -b_1\] so \[f \circ (\alpha_1 -\zeta \circ \gamma_1)=b_1\] and \[f \circ \zeta \circ \gamma_2=f \circ \beta_1-b_2\] then \[f \circ (\beta_1 - \zeta \circ \gamma_2)=b_2\] and, the morphism \[(\alpha_1-\zeta \circ \gamma_1 ~ \beta_1-\zeta \circ \gamma_2):A \oplus I_0 \to X\] make both triangles commute.

The case $0 \to M$ is left to the reader.

Now, let us show the inverse inclusion. We have immediately $I^\square \subseteq J^\square$ since $J \subseteq I$. Let $f \in I^\square$. Let us show that, using the notations explained in section \ref{sec:not}, $g \in \overline{\M}^\perp$ and $\tilde{g} \in \overline{\M}^\perp$. First, let $b:M \to Y$ be a morphism, with $M \in \M$. As $f \in I^\square$, there exists $\varphi:M \to X$ such that $f \circ \varphi=b$ then $g \circ b=g \circ f \circ \varphi$ and this morphism factorizes through an injective. This shows that $g \in \overline{\M}^\perp$.
Second, Let $K$ be the kernel of $f$. Let $b:M \to K$. Let $N$ be an $\M$-approximation of $X$. We denote it by $a:N \to X$. As $a$ is an $\M$-approximation, there exists $h:M \to N$ such that $\tilde{g} \circ b=a \circ h$. Let $\iota_M:M \to I_M$ be the canonical injection. As $\begin{pmatrix} h \\ \iota_M \end{pmatrix}$ is an inflation, we introduce $C$ its cokernel. Let $(k~u):N \oplus I_M \to C$. If we put $(a~0):N \oplus I_M \to X$, then there exists $r:C \to Y$ such that there is a morphism of short exact sequences:
\[\xymatrix{
0\ar[d] & 0\ar[d] \\
M \ar_{\left(^h_{\iota_M}\right)}[d] \ar^{b}[r] & K \ar^{\tilde{g}}[d] \\
N \oplus I_M \ar_{(k~u)}[d] \ar^{(a~0)}[r] & X \ar^{f}[d] \\
C \ar[d] \ar^{r}[r] & Y \ar[d] \\
0 & 0}\]
We are going to build an element of $I$ from $N \oplus I_M \to C$. We add $J$ the injective envelope of $N \oplus I_M$ ($J=I_M \oplus I_N$). Then the morphism \[\begin{pmatrix} k & u \\ \iota_1 & \iota_2 \end{pmatrix} : N \oplus I_M \to C \oplus J\] is in $I$. Then the following square is commutative
\[\xymatrix{N \oplus I_M \ar[r] \ar[d] & X \ar[d] \\ C \oplus J \ar[r] & Y }\] and as $f \in I^\square$, there exist $(\varphi_1 ~ \varphi_2):C \oplus J \to X$ such that $f \circ \varphi_1 = r$ and $f \circ \varphi_2=0$ and \[(\varphi_1 \circ k + \varphi_2 \circ \iota_1 ~ \varphi_1 \circ u + \varphi_2 \circ \iota_2)=(a~0).\]
Then we have
\begin{equation}
\begin{aligned} \tilde{g} \circ b & = a \circ h \\
& = (a~0) \circ \begin{pmatrix} h \\ \iota_M \end{pmatrix} \\
& = (\varphi_1 ~ \varphi_2) \circ \begin{pmatrix} k & u \\ \iota_1 & \iota_2 \end{pmatrix} \circ \begin{pmatrix} h \_ \iota_M \end{pmatrix} \\
& = \begin{pmatrix}\varphi_1 \circ k + \varphi_2 \circ \iota_1 ~ \varphi_1 \circ u + \varphi_2 \circ \iota_2\end{pmatrix} \circ \begin{pmatrix} h \\ \iota_M \end{pmatrix} \\
& = (\varphi_1~\varphi_2) \circ \begin{pmatrix} 0 \\ \iota_1 \circ h + \iota_2 \circ \iota_M. \end{pmatrix} \end{aligned}
\end{equation}
This shows that $\tilde{g} \circ b$ factorizes through an injective, thus $\tilde{g} \in \overline{\M}^\perp$.

Then we have shown that $I^\square=J^\square \cap \W$.
\end{proof}

In addition to the first factorization built in Lemma \ref{cor:fact}, we also have a second factorization:

\begin{lem}\label{lem:fac2}
Any morphism $f:X\to Y$, where $X \in \mathrm{pr}\M$ can be factorized through a morphism in $^\square (I^\square)$ followed by a morphism in $I^\square$.
\end{lem}

\begin{proof}
Let us show that any morphism $f:X\to Y$, where $X \in \mathrm{pr}\M$ can be factorized through a morphism in $^\square (I^\square)$ followed by a morphism in $I^\square$.

Let $f:X \to Y$ be a morphism, such that $X \in \mathrm{pr}\M$.
Let \[0\to M_1 \to M_0 \to X \to 0\] be a short exact sequence, where $c$ denotes the morphism $M_0 \to X$. From lemma \ref{lem:prmapp}, there exists a $\mathrm{pr}\M$-approximation $a:A\to Y$, with $A \in \mathrm{pr}\M$. Let $\varepsilon:X \to \mho M_1$ be the induced morphism in the short exact sequence \[0 \to M_0 \to X \oplus I_{M_0} \to \mho M_1 \to 0. \]
As $A \to Y$ is a $\mathrm{pr}\M$-approximation, and $X \in \mathrm{pr}\M$,  there exists $r:X \to A$ such that \[f=a \circ r.\] We are going to show that the factorization\[ \xymatrix{X \ar^f[rr] \ar_{\left(^r_\varepsilon\right)}[dr] & & Y \\ & A \oplus \mho M_1 \ar_{(a~0)}[ur] & }\] satisfies the required properties.

It is immediate that $(a~0)$  lies in $J^\square \cap \W$. By Lemma \ref{lem:ega}, it is an element of $I^\square$. Let us show that $\begin{pmatrix} r\\ \varepsilon \end{pmatrix} \in ~^\square(I^\square)$.

Let \[\xymatrix{X \ar^{\alpha}[r] \ar_{\left(^r_\varepsilon\right)}[d] & U \ar^h[d] \\ A \oplus \mho M_1 \ar_{~~~~(\beta_1 ~ \beta_2)}[r] & V}\] be a commutative square with $h \in I^\square$.

The object $A \oplus \mho M_1$ is in $\mathrm{pr}\M$: by lemma \ref{lem:cof}, it is cofibrant. Thus, there exists a morphism $(\varphi_1~\varphi_2):A\oplus \mho M_1 \to U$ such that \begin{equation}\label{eq:fac23} h \circ \varphi_1=\beta_1 \text{ and } h \circ \varphi_2=\beta_2. \end{equation}

Unfortunately, this morphism is not the good candidate, because it does not make the upper triangle commute. We have to modify it.

As $h \in J^\square \cap \W$, it is a deflation. Let $u:K \to X$ be its kernel. There exists $\gamma:X \to K$ such that \begin{equation}
\label{eq:fac4} \alpha= u \circ \gamma + \varphi_1 \circ r + \varphi_2 \circ \varepsilon. \end{equation}
Moreover, we know that $u \in \overline{\M}^\perp$. So the morphism $u \circ \gamma \circ c$ factorizes through an injective module $I_{M_1}$.  It follows that there exists $\beta$ as in the following diagram:

\[ \xymatrix{
M_1 \ar@{>->}[r] \ar@{=}[d] & M_0 \ar@{->>}^c[r] \ar@{>->}[d] & X \ar@{>->}^{\varepsilon}[d] \ar@/^1pc/^{u \circ \gamma}[ddr] & \\
M_1 \ar@{>->}[r] & I_{M_1} \ar@{->>}[r] \ar@/_1pc/_{\beta}[drr] & \mho M_1 \ar^{\exists a_2}[dr] & \\
& & & C} \]

From the push-out property, there exists a morphism $a_2:\mho M_1 \to C$ such that \begin{equation} \label{eq:fac5} u \circ \gamma = a_2 \circ \varepsilon. \end{equation}

Then, from equations (\ref{eq:fac4}) and (\ref{eq:fac5}), we have $(\varphi_1~~\varphi_1+a_2) \circ \begin{pmatrix} r \\ \varepsilon \end{pmatrix}=\alpha$.

Unfortunately again, this morphism is not the good candidate. Now, the upper triangle commutes, but we have lost the commutativity of the lower one. We have to modify it one more time.

Moreover, the morphism $\begin{pmatrix} r \\ \varepsilon \end{pmatrix}$ is an inflation. Let $(c_1~c_2):A \oplus \mho M_1 \to C$ be a cokernel for $\begin{pmatrix} r \\ \varepsilon \end{pmatrix}$.
We have \[ h \circ (\varphi_1 ~ \varphi_2 + a_2) \circ \begin{pmatrix} r \\ \varepsilon \end{pmatrix}= h \circ \alpha= (\beta_1 ~ \beta_2) \circ \begin{pmatrix} r \\ \varepsilon \end{pmatrix}.\]
The morphism $h \circ (\varphi_1~~\varphi_2+a_2) - (\beta_1 ~ \beta_2)$ factorizes through the cokernel $C$ of $\begin{pmatrix} r \\ \varepsilon \end{pmatrix}$: there exists $b:C \to V$ such that \begin{equation}\label{eq:fac6} (\beta_1~\beta_2)=h \circ (\varphi_1 ~ \varphi_2+a_2)-b \circ (c_1~c_2). \end{equation}
Moreover, $C \in \mathrm{pr}\M$ by lemma \ref{lem:pr}, so that it is cofibrant (see Lemma \ref{lem:cof}), as $h \in J^\square \cap \W$, there exists $d:C \to U$ such that \begin{equation}\label{eq:fac7} h \circ d=b.\end{equation}

We now have enough information in order to choose the good candidate for the morphism $A \oplus \mho M_1 \to U$ which makes both triangles commute. It is:
\[(\varphi_1-d \circ c_1 ~~~ \varphi_2 + a_2 - d \circ c_2).\]
Indeed, for the upper triangle:
\begin{equation} \begin{aligned} (\varphi_1-d \circ c_1~~~\varphi_2 + a_2-d \circ c_2) \circ \begin{pmatrix} r \\ \varepsilon \end{pmatrix} & = \varphi_1 \circ r - d \circ c_1 \circ r+ \varphi_2 \circ \varepsilon + a_2 \circ \varepsilon - d \circ c_2 \circ \varepsilon \\
& = \alpha-u \circ \gamma - \varphi_2 \circ \varepsilon - d \circ c_1 \circ r + \varphi_2 \circ \varepsilon \\
& + a_2 \circ \varepsilon - d \circ c_2 \circ \varepsilon \text{ by (\ref{eq:fac4})} \\
& = \alpha- a_2 \circ \varepsilon - d \circ c_1 \circ r + a_2 \circ \varepsilon - d \circ c_2 \circ \varepsilon \text{ by (\ref{eq:fac5})} \\
& = \alpha- d \circ (c_1~c_2) \circ \begin{pmatrix} r \\ \varepsilon \end{pmatrix} \\
& = \alpha. \end{aligned} \end{equation}
The upper triangle commutes.

Now let us show the commutativity of the lower triangle:
\begin{equation} \begin{aligned} h \circ (\varphi_1-d \circ c_1~~~\varphi_2 + a_2-d \circ c_2) & = (h \circ \varphi_1 - h \circ d \circ c_1 ~~~ h \circ \varphi_2 + h \circ a_2 - h \circ d \circ c_2) \\
& = (h \circ \varphi_1-b \circ c_1~~~ h \circ \varphi_2 + h \circ a_2 - b \circ c_2) \text{ by (\ref{eq:fac7})} \\
& = (h \circ \varphi_1 + \beta_1 - h \circ \varphi_1~~~\beta_2 + b \circ c_2 - b \circ c_2) \text{ by (\ref{eq:fac6})} \\
& = (\beta_1~\beta_2) \end{aligned} \end{equation}
Therefore, both triangles commute, so $\begin{pmatrix} r \\ \varepsilon \end{pmatrix} \in ~^\square(I^\square)$, and we have the factorization when the domain of the morphism is cofibrant.
\end{proof}

\subsection{The almost model structure in the case of Frobenius categories}

We are going to show the following theorem:

Let \[J= \{ f : 0 \to \mho M, M \in \M \}\] and \[I=\{ f : M_0 \to X \oplus I_0, X \in \mathrm{pr}\M \} \cup \{0 \to M,M \in \M \}.\]
Let $G$ be the functor

\[\begin{array}{ccccc}
G & : & \E & \to & \mathrm{Mod} \overline{\M} \\
 & & X & \mapsto & \overline{\E}(-,X)/\overline{\M} \\
\end{array}\]

Let \[\W=\{f, Gf\text{ is an isomorphism} \}.\]

\begin{theo}
Let $\E$ be a weakly idempotent complete Frobenius category. Assume that $\M \subseteq \E$ is a rigid, contravariantly finite (strictly) full subcategory of $\E$ containing all the injective objects, and stable under taking direct sums and summands.

Let $J^\square$ be the class of fibrations. The cofibrations are given by the left-lifting property from acyclic fibrations. Then, $(Fib, Cof, \W)$ nearly form a model structure for the category $\E$. Indeed, the second factorization is found only when the domain is cofibrant.
\end{theo}

\begin{proof}
We show the conditions required in order to apply the version of Hovey in \cite[Theorem 2.1.19]{Ho} described in Theorem \ref{th:hov}.
\begin{enumerate}[label=(\roman*)]
\item It is well-known that $\W$ is stable under retracts. Indeed, the class of morphisms in any category satisfies $(i)$ and $(ii)$.

%

\item It is also well-known that $\W$ satisfies the "$2$ out of $3$" property.

%
%

\item $^\square(J^\square) \subseteq \W \cap~^\square(I^\square)$.

As $J \subseteq I$, we automatically have that $^\square(J^\square) \subseteq~^\square(I^\square)$.
From lemma \ref{lem:jperpperp} and $(i)$, it is immediate that any morphism of $^\square(J^\square)$ is a weak equivalence.

\item $I^\square = J^\square \cap \W$.

This is exactly Lemma \ref{lem:ega}.

\item The first required factorization is exactly corollary \ref{cor:fact}, part 1.
\item The second required factorization is exactly lemma \ref{lem:fac2}.
\end{enumerate}
\end{proof}

Recall that $\mathrm{mod}\overline{\M}$ is the classe of finitely generated modules over $\overline{\M}$.

As a consequence, we have directly the theorem of Quillen for Frobenius categories.

\begin{theo}
Let $\mathrm{Ho}~\E$ be the localization of $\E$ at the class $\W$. There is an equivalence of categories \[\mathrm{Ho}~\E \simeq \mathrm{mod}~ \overline{\M}. \]
\end{theo}

\bibliographystyle{alpha}
\bibliography{biblio}

\end{document}